\documentclass[reqno,a4paper]{amsart}

\usepackage[utf8]{inputenc}
\usepackage[T1]{fontenc}
\usepackage[british]{babel}
\usepackage{amssymb,amsthm,amsmath}
\usepackage{mathtools}
\usepackage{calrsfs}
\usepackage{mathbbol}
\usepackage{mathabx}
\usepackage{graphicx}
\usepackage[all]{xy}
\usepackage{enumitem}
\usepackage{booktabs}
\usepackage{xcolor}
\usepackage[colorlinks]{hyperref}

\allowdisplaybreaks

\theoremstyle{plain}
\newtheorem{lemma}[subsection]{Lemma}
\newtheorem{proposition}[subsection]{Proposition}
\newtheorem{theorem}[subsection]{Theorem}
\newtheorem{corollary}[subsection]{Corollary}

\theoremstyle{remark}
\newtheorem{example}[subsection]{Example}
\newtheorem{remark}[subsection]{Remark}

\theoremstyle{definition}
\newtheorem{definition}[subsection]{Definition}
\newtheorem{notation}[subsection]{Notation}
\newtheorem{convention}[subsection]{Convention}

\newcommand{\Cech}{\ensuremath{\check C}}
\newcommand{\HH}{\ensuremath{\check H}}
\newcommand{\Tsjech}{\v Cech}
\newcommand{\defn}{\textbf}
\newcommand{\del}{\partial}
\newcommand{\Diag}{\ensuremath{\Delta}}
\newcommand{\noproof}{\hfill \qed}
\renewcommand{\square}{\raisebox{.4mm}{\,\ensuremath{\boxvoid}}}
\newcommand{\To}{\Rightarrow}
\newcommand{\two}{\rightrightarrows}

\DeclareMathOperator{\ab}{Ab}
\DeclareMathOperator{\Cod}{Cod}
\DeclareMathOperator{\coeq}{coeq}
\DeclareMathOperator{\coker}{coker}
\DeclareMathOperator{\colim}{colim}
\DeclareMathOperator{\Dom}{Dom}
\DeclareMathOperator{\Eq}{Eq}
\DeclareMathOperator{\Ker}{Ker}
\DeclareMathOperator{\Nerve}{Ner}
\newcommand{\R}[1]{\Eq(#1)}

\newcommand{\A}{\ensuremath{\mathbb{A}}}
\newcommand{\B}{\ensuremath{\mathbb{B}}}
\newcommand{\C}{\ensuremath{\mathbb{C}}}
\newcommand{\D}{\ensuremath{\mathbb{D}}}
\newcommand{\X}{\ensuremath{\mathbb{X}}}

\newcommand{\E}{\ensuremath{\mathcal{E}}}
\newcommand{\Z}{\ensuremath{\mathcal{Z}}}

\newcommand{\Ab}{\ensuremath{\mathsf{Ab}}}
\newcommand{\Grp}{\ensuremath{\mathsf{Grp}}}
\newcommand{\Arr}{\ensuremath{\mathsf{Arr}}}
\newcommand{\Cat}{\ensuremath{\mathsf{Cat}}}
\newcommand{\CExt}{\ensuremath{\mathsf{CExt}}}
\newcommand{\Ext}{\ensuremath{\mathsf{Ext}}}
\newcommand{\Grpd}{\ensuremath{\mathsf{Grpd}}}
\newcommand{\RG}{\ensuremath{\mathsf{RG}}}
\newcommand{\s}{\ensuremath{\mathsf{s}}}
\newcommand{\Set}{\ensuremath{\mathsf{Set}}}

\def\pullback{
 \ar@{-}[]+R+<6pt,-1pt>;[]+RD+<6pt,-6pt>%
 \ar@{-}[]+D+<1pt,-6pt>;[]+RD+<6pt,-6pt>}

\newdir{>>}{{}*!/3.5pt/:(1,-.2)@^{>}*!/3.5pt/:(1,+.2)@_{>}*!/7pt/:(1,-.2)@^{>}*!/7pt/:(1,+.2)@_{>}}
\newdir{ >>}{{}*!/8pt/@{|}*!/3.5pt/:(1,-.2)@^{>}*!/3.5pt/:(1,+.2)@_{>}}
\newdir{ |>}{{}*!/-3.5pt/@{|}*!/-8pt/:(1,-.2)@^{>}*!/-8pt/:(1,+.2)@_{>}}
\newdir{ >}{{}*!/-8pt/@{>}}
\newdir{>}{{}*:(1,-.2)@^{>}*:(1,+.2)@_{>}}
\newdir{<}{{}*:(1,+.2)@^{<}*:(1,-.2)@_{<}}

\begin{document}

\author{F.~Afsa}
\author{T.~Van~der Linden}

\email[F.~Afsa]{florent.afsa@uclouvain.be}
\email[T.~Van~der Linden]{tim.vanderlinden@uclouvain.be}
\email[T.~Van~der Linden]{tim.van.der.linden@vub.be}

\address[F.~Afsa, T.~Van~der Linden]{Institut de Recherche en Math\'ematique et Physique, Universit\'e catholique de Louvain, che\-min du cyclotron~2 bte~L7.01.02, B--1348 Louvain-la-Neuve, Belgium}
\address[T.~Van~der Linden]{Mathematics \& Data Science, Vrije Universiteit Brussel, Pleinlaan 2, B--1050 Brussel, Belgium}

\thanks{The second author is a Research Director of the Fonds de la Recherche Scientifique--FNRS. He wishes to thank Dr~Viji J.~Thomas and the Indian Institute of Science Education and Research at Thiruvananthapuram, and Prof.~Manuel Ladra and the University of Santiago de Compostela for their kind hospitality.}

\date{}

\title[Higher \v Cech derived functors via Galois theory]{Higher \v Cech derived functors via Galois theory}

\begin{abstract}
    We explore sufficient conditions for the higher \Tsjech\ derived functors of a Birkhoff reflector to agree with its ordinary, Barr--Beck comonadic homology. Using categorical Galois theory together with a spectral-sequence argument, we prove that this is the case for the reflector from any semi-abelian variety onto a subvariety: for every \(n\geq 1\), the \(n\)th \Tsjech\ derived functor of the reflector coincides with its \((n+1)\)st comonadic homology object. More generally, the comparison holds for any Birkhoff reflector of a semi-abelian category equipped with a comonad satisfying Quillen's conditions.
\end{abstract}

\maketitle

\section{Introduction}

In his foundational article~\cite{Hopf}, Hopf expressed the second integral homology group of a group \(G\) in purely algebraic terms: if \(G\) is presented as a quotient \(F/R\) of a free group \(F\) by a normal subgroup \(R\), then
\[
    H_{2}(G)\cong \frac{R\cap [F,F]}{[F,R]}\text{,}
\]
the right-hand side being independent of the chosen presentation. In~\cite{Brown-Ellis}, Brown and Ellis extended this result to all degrees: the homology group \(H_{n+1}(G)\) admits an analogous description in terms of an \(n\)-fold presentation of \(G\), with an intersection of commutators in the denominator. Their proof is topological.

An algebraic proof of these higher Hopf formulae was found by Donadze, Inassaridze and Porter~\cite{Donadze-Inassaridze-Porter}. The key concept in their approach is that of a \Tsjech\ derived functor: rather than deriving a functor by applying it to a simplicial resolution~\cite{Barr-Beck,Tierney-Vogel2}, one applies it to the \Tsjech\ complexes induced by an \(n\)-fold presentation, and takes the homology of the diagonal. The name is not incidental, and neither the term nor the notion originates in~\cite{Donadze-Inassaridze-Porter}. The \Tsjech\ complex of a presentation \(f\colon{X\to Y}\) is the nerve of its kernel pair (Section~\ref{section reflection nerve}), so that its object of \(n\)-simplices is the \((n+1)\)-fold fibre power of \(f\); the fibre powers play here the part played by the \((n+1)\)-fold intersections in the nerve of an open covering, the construction by which Aleksandrov and \Tsjech, in the late 1920s, carried homology beyond simplicial complexes. \Tsjech\ derived functors of a functor on the category of groups were developed by Inassaridze~\cite{IH,Inasaridze75}, and had already been used to obtain a Hopf-type formula for the second Conduch\'e--Ellis homology of a precrossed module (Inassaridze and Khmaladze~\cite{IK}). What~\cite{Donadze-Inassaridze-Porter} contributes is the \(n\)-fold version, together with the Hopf formulae it yields. Their main computational tool is a result of Loday~\cite{Loday} which calculates the homology of the diagonal of the multiple nerve of an \(n\)-fold groupoid.

On the other hand, a categorical proof, valid beyond the category of groups, was given in~\cite{EGVdL}: in a semi-abelian category~\cite{Janelidze-Marki-Tholen} with enough projectives, the homology objects arising from any Birkhoff reflector \(F\)---in the varietal case, comonadic homology in the sense of Barr and Beck~\cite{Barr-Beck}---admit Hopf formulae in all degrees, obtained by interpreting them through categorical Galois theory~\cite{Janelidze-Kelly, Janelidze:Hopf} in terms of higher central extensions.

The Galois-theoretic reading of the Hopf formulae goes back to Janelidze. Categorical Galois theory was introduced in~\cite{Janelidze:Pure}---see~\cite{Borceux-Janelidze} for a systematic treatment---and applied by Janelidze and Kelly~\cite{Janelidze-Kelly} to give a categorical account of the relative central extensions of Fr\"ohlich~\cite{Frohlich}, Lue~\cite{Lue} and Furtado-Coelho~\cite{Furtado-Coelho}. Answering a question of Ronald Brown, Janelidze then introduced the double central extensions of groups and characterised them by means of commutators~\cite{Janelidze:Double}; his lecture~\cite{Janelidze:Hopf-talk} proposed to iterate this step into a tower of higher central extensions and to read the Brown--Ellis formulae off it, a programme carried out in the semi-abelian setting in~\cite{EGVdL}. His characterisation of double central extensions was in the meantime extended to Mal'tsev varieties by Gran and Rossi~\cite{Gran-Rossi}, to semi-abelian categories in~\cite{RVdL}, and, beyond the semi-abelian context, to exact Mal'tsev categories in~\cite{EverVdL3}.

Hopf formulae in higher degrees are by now available in a range of settings. Besides the original case of groups~\cite{Conrad,Rodicio,Stohr,Brown-Ellis,Donadze-Inassaridze-Porter}, they have been obtained for homotopy \((n+1)\)-types in the shape of crossed \(n\)-cubes~\cite{CIKL-nTypes}, for Leibniz \(n\)-algebras~\cite{CKL}, for precrossed modules with coefficients in crossed modules~\cite{EGVdL}, for \(n\)-fold internal groupoids in a semi-abelian category~\cite{EG-honfg}, for compact groups and other situations in which the reflector is protoadditive~\cite{Everaert-Gran-TT}, and recently for skew braces~\cite{GranLetourmyVendramin2025}. Everaert~\cite{EverHopf} gives a direct proof that the higher Hopf formulae do not depend on the chosen \(n\)-fold presentation, and~\cite{EverMaltsev,RenaudVdL} carry the theory of higher central extensions on which they rest beyond the semi-abelian context.

Thus two general methods lead to the same Hopf formulae: one through \Tsjech\ derived functors, the other through Galois theory. \emph{How exactly are the two approaches related?} Answering that question is the aim of this article. We prove Theorem~\ref{Theorem Comparison Level n}: for the reflector \(F\) from a semi-abelian variety \(\C\) to a subvariety~\(\X\), the higher \Tsjech\ derived functors of \(F\) coincide with its ordinary homology---for every \(n\geq 1\), the \(n\)th \Tsjech\ derived functor of \(F\) at \(Y\) is isomorphic to the \((n+1)\)st homology object \(H_{n+1}(Y,F)\) of \(Y\) with coefficients in \(F\), the degree shift being the usual one between the simplicial numbering of derived functors and the group homology numbering of~\cite{EGVdL}. The hypothesis that \(\C\) is a variety is what makes these comonadic homology objects available, and it may be weakened accordingly: the comparison holds whenever \(\C\) is a semi-abelian category carrying a comonad which satisfies Quillen's conditions (Theorem~\ref{Theorem Comparison General}); the underlying Galois-theoretic identification of the fundamental group with the first \Tsjech\ derived functor (Theorem~\ref{Theorem Comparison Level 1}) holds more generally still, for any Birkhoff reflector of a semi-abelian category with enough projectives. The agreement between the two families of Hopf formulae is thereby explained at a structural level: both approaches compute the same derived functors.

The proof combines two main ingredients. The first is Galois-theoretical: the fundamental group of an object \(Y\), defined by means of a weakly universal central extension, may be computed directly from the kernel pair of any projective presentation of \(Y\), without centralising (Theorem~\ref{Theorem Interpretation Fundamental Group}); this identifies it with the first \Tsjech\ derived functor of \(F\) at \(Y\) (Theorem~\ref{Theorem Comparison Level 1}). The second is simplicial: a categorical version of Loday's theorem (Theorem~\ref{Loday Theorem}), which we prove by means of Quillen's spectral sequences for double simplicial objects~\cite{DGQlln1966}, available in any semi-abelian category by~\cite{Afsa-Spectral}. An induction along the tower of Galois structures of~\cite{EGVdL} then yields the comparison in all degrees.

Let us set out, result by result, what is new here and what we recall from elsewhere. The higher Hopf formulae and the tower of Galois structures along which we perform our proof by induction are due to Everaert, Gran and Van der Linden~\cite{EGVdL}; the \Tsjech\ derived functors and the group-theoretic form of Loday's theorem go back to Donadze, Inassaridze and Porter~\cite{Donadze-Inassaridze-Porter}; and the degree-one comparison \(\HH_{1}F(Y)\cong H_{2}(Y,F)\) was first proved---by a direct simplicial argument independent of ours, and in the generality of Theorem~\ref{Theorem Comparison General}---by Culot and Everaert~\cite{Culot-Everaert}, who leave the higher degrees open; our overlap with their work is confined to that single degree, and Remark~\ref{Remark Culot Everaert} sets out how the two accounts differ. That the Galois groupoid, and hence the fundamental group, may be computed from a projective presentation without first centralising (Lemma~\ref{Lemma Simple Galois Groupoid}, Theorem~\ref{Theorem Interpretation Fundamental Group}) is in essence already contained in Bourn and Rodelo~\cite{Bourn-Rodelo2}, as recorded in Remark~\ref{Remark No Centralisation}; what Section~\ref{section fundamental group} contributes is the explicit centralisation-free form the induction needs. The genuinely new results are the categorical version of Loday's theorem (Theorem~\ref{Loday Theorem}), valid in any semi-abelian category, and the inductive comparison of Section~\ref{section induction}, which carries the degree-one isomorphism up the entire tower to yield \(\HH_{n}F(Y)\cong H_{n+1}(Y,F)\) for all \(n\geq 1\).

\subsection*{Structure of the paper}
Section~\ref{section basic concepts} recalls the needed concepts from categorical Galois theory, the construction of the centralisation functor, and the behaviour of limits and colimits in categories of extensions. Section~\ref{section fundamental group} studies the fundamental group of an object with respect to a Birkhoff reflector \(F\), characterises the extensions inverted by \(F\), and shows that the Galois groupoid of a projective presentation may be computed without centralising. Section~\ref{section reflection nerve} treats nerves of internal groupoids and \Tsjech\ complexes, culminating in the comparison between the fundamental group and the first \Tsjech\ derived functor. Section~\ref{section higher galois} recalls the tower of higher Galois structures and introduces \(n\)-fold groupoids and their nerves. Section~\ref{section loday} establishes the categorical version of Loday's theorem. Section~\ref{section induction} contains the inductive argument leading to the main result, Theorem~\ref{Theorem Comparison Level n}.

\section{Basic concepts}\label{section basic concepts}

The following summarises basic concepts and results from categorical Galois theory, as developed in~\cite{Janelidze-Kelly} and~\cite{EGVdL}.

\begin{definition}
    A \defn{Galois structure} \(\Gamma=(\A,\B,\E,\Z,I,H)\) consists of two categories \(\A\) and \(\B\), an adjunction
    \[
        \xymatrix{\A \ar@<.5ex>[r]^-{I} & \B \ar@<.5ex>[l]^-{H}}
    \]
    with unit \(\eta\colon{1_{\A}\to HI}\) and counit \(\varepsilon\colon{IH\to 1_{\B}}\), and two classes of morphisms \(\E\) in~\(\A\) and \(\Z\) in \(\B\) such that:
    \begin{itemize}
        \item \(\A\) has pullbacks along morphisms in \(\E\);
        \item \(\E\) and \(\Z\) contain all isomorphisms and are stable under pullbacks and under composition;
        \item \(H(\Z)\subseteq \E\);
        \item \(I(\E)\subseteq \Z\);
        \item \(\varepsilon\colon IH\to 1_{\B}\) is a natural isomorphism;
        \item \(\eta_{A}\colon{A\to HI(A)}\) is in \(\E\) for all \(A\) in \(\A\).
    \end{itemize}
    A map in \(\E\) is called an \defn{extension}.
\end{definition}

For each \(A\) in \(\A\), let us write \(\E(A)\) for the full subcategory of the slice category \((\A\downarrow A)\) determined by the extensions with codomain \(A\). For an arrow \(a\colon{A\to B}\) in~\(\A\), we have a functor \(a^{*}\colon{\E(B)\to \E(A)}\) defined by pulling back along \(a\). For every \(A\) in \(\A\), we have an adjunction
\[
    \xymatrix{\E(A) \ar@<.5ex>[r]^-{I^A} & \Z(I(A)) \ar@<.5ex>[l]^-{H^A}}
\]
defined by \(I^{A}(f\colon{X\to A})=I(f)\) and \(H^{A}(g\colon{Y\to I(A)})=\eta_A^*H(g)\). The unit and counit of this adjunction are denoted by \(\eta^{A}\) and \(\varepsilon^{A}\), respectively.

\begin{definition}
    The Galois structure \(\Gamma\) is \defn{admissible} if for every \(A\) in \(\A\) the functor \(H^{A}\) is fully faithful; equivalently, if \(\varepsilon^{A}\) is a natural isomorphism for every~\(A\) in \(\A\).
\end{definition}

Let \(\C\) be a Barr-exact category~\cite{Barr-Grillet-vanOsdol}. Write \(\Arr(\C)\) for the category of arrows in \(\C\), and \(\Ext(\C)\) for the full subcategory of \(\Arr(\C)\) determined by the regular epimorphisms in \(\C\). An \defn{extension} in \(\C\) is simply a regular epimorphism.

\begin{definition}\label{Definition Birkhoff Subcategory}
    A \defn{Birkhoff subcategory} \(\X\) of \(\C\) is a full and replete reflective subcategory closed under subobjects and regular quotients. A reflector \(F\colon \C\to \X\) is called a \defn{Birkhoff reflector}.
\end{definition}

Let \(\E\) be the class of extensions in \(\C\).

\begin{definition}\label{Definition Double Extension}
    A \defn{double extension} in \(\C\) is a commutative square
    \[
        \xymatrix{X \ar[d]_-{f} \ar[r]^-{g} & Y \ar[d]^-{f'} \\
        Z \ar[r]_-{g'} & W}
    \]
    in which \(f\), \(f'\), \(g\) and \(g'\) are extensions, and the induced morphism from \(X\) to the pullback of \(f'\) and \(g'\) is also an extension.
\end{definition}

Every double extension is a pushout square. In a semi-abelian category, every pushout square of extensions is a double extension (essentially by~\cite{Carboni-Kelly-Pedicchio}).

\begin{definition}
    A full and replete reflective subcategory \(\X\) of \(\C\), with reflector \(F\colon{\C\to \X}\) and unit \(\eta\), is a \defn{strongly \(\E\)-Birkhoff subcategory} if for every extension \(f\colon{X\to Y}\) in \(\C\) the induced square
    \begin{equation}\label{Birkhoff Square}
        \vcenter{\xymatrix{X \ar[d]_-{\eta_{X}} \ar[r]^-{f} & Y \ar[d]^-{\eta_{Y}} \\
        F(X) \ar[r]_-{F(f)} & F(Y)}}
    \end{equation}
    ---the \defn{strong \(\E\)-Birkhoff square} of \(f\)---is a double extension.
\end{definition}

From now on, suppose \(\C\) is a \defn{semi-abelian} category~\cite{Janelidze-Marki-Tholen}: pointed, Barr exact and Bourn protomodular, with binary coproducts. Groups, Lie algebras, crossed modules and, more generally, all varieties of \(\Omega\)-groups in the sense of Higgins~\cite{Higgins} are of this kind. Several arguments below need only the weaker \defn{homological} context: regular instead of exact, and without coproducts. We freely use the calculus these conditions support---kernels and cokernels, exact sequences, the Short Five and \(3\times 3\) Lemmas---for which we refer to Borceux and Bourn~\cite{Borceux-Bourn}.

Let \(\X\) be a strongly \(\E\)-Birkhoff subcategory of \(\C\) with reflector \(F\colon \C\to \X\). In a semi-abelian category, which is exact and Mal'tsev, the strongly \(\E\)-Birkhoff subcategories are exactly the Birkhoff ones~\cite[Section~2]{EGVdL}; from Section~\ref{section fundamental group} onwards we accordingly speak simply of a \defn{Birkhoff reflector}. Write \(\Z\) for the class of regular epimorphisms in \(\X\), and \(H\colon \X\to \C\) for the inclusion functor. Then \(\Gamma=(\C,\X,\E,\Z,F,H)\) is an admissible Galois structure~\cite[Proposition 2.6]{EGVdL}. We suppress \(H\) when no ambiguity arises. An extension \(f\colon{X\to Y}\) in \(\C\) is called \defn{trivial} (with respect to \(\Gamma\)) if its strong \(\E\)-Birkhoff square~\eqref{Birkhoff Square} is a pullback, \defn{central} if its pullback along some regular epimorphism \(y\colon{Z\to Y}\) is trivial, and \defn{normal} if its pullback along itself is trivial. The full subcategory of \(\Ext(\C)\) determined by the central extensions is denoted by \(\CExt(\C)\).

\begin{lemma}\label{split central extension}
    In a semi-abelian category:
    \begin{itemize}
        \item Every trivial extension is central, and every normal extension is central;
        \item \cite[Theorem~4.8]{Janelidze-Kelly} Every central extension which is split is trivial, and every central extension is normal---a semi-abelian category being in particular Goursat.\noproof
    \end{itemize}
\end{lemma}

The inclusion functor \({\subseteq_{1}}\colon{\CExt(\C)\to \Ext(\C)}\)---written thus, as in~\cite{EGVdL}, to keep the letter \(H\) free for homology---has a left adjoint \(F_{1}\), called the \defn{centralisation functor}~\cite[4.2]{EGVdL}. To construct it, let \(G\) be the functor defined by \(G(X)\coloneq\Ker(\eta_X\colon X\to F(X))\), with \(\mu_X\coloneq\ker(\eta_X)\colon G(X)\to X\). Now consider, for \(f\colon{X\to Y}\) and its kernel pair \((\R{f},\pi_{0},\pi_{1})\), the diagram
\[
    \xymatrix@C=4em{\Ker(G(\pi_0)) \ar[r]^-{\ker(G(\pi_0))} \ar[d] & G(\R{f}) \ar[d]_-{\mu_{\R{f}}} \ar@<.5ex>[r]^-{G(\pi_{0})} \ar@<-.5ex>[r]_-{G(\pi_{1})} & G(X) \ar[d]^-{\mu_{X}} \\
    \Ker(\pi_0) \ar[r]_-{\ker(\pi_0)} & \R{f} \ar@<.5ex>[r]^-{\pi_{0}} \ar@<-.5ex>[r]_-{\pi_{1}} & X.}
\]
Let \(G_1(f)\) be the composite \(\mu_X\circ G(\pi_1)\circ \ker(G(\pi_0))\colon G_1[f]\coloneq\Ker(G(\pi_0))\to X\) and write \(\rho^{1}_{f}\coloneq\coker(G_1(f))\). Then \(G_1(f)\) is the kernel of \(\rho^{1}_{f}\), and we have a factorisation
\[
    \xymatrix{X \ar[rr]^-{f} \ar[dr]_-{\rho^{1}_{f}} & & Y \\
    & F_{1}[f] \ar[ur]_-{F_{1}(f)} & }
\]
where \((\rho^{1}_{f},1_{Y})\) is the \(f\)-component of the unit of the adjunction \(F_{1}\dashv{\subseteq_{1}}\). By naturality of \(\mu\) and \(f\pi_{0}=f\pi_{1}\),
\[
    f\circ G_{1}(f)=\mu_{Y}\circ G(f\pi_{1})\circ \ker(G(\pi_{0}))=\mu_{Y}\circ G(f)\circ G(\pi_{0})\circ \ker(G(\pi_{0}))=0\text{.}
\]
In particular, \(G_{1}[f]=\Ker(\rho^{1}_{f})\leq\Ker(f)\).

The subcategory \(\CExt(\C)\) happens to be a strongly \(\E^{1}\)-Birkhoff subcategory of \(\Ext(\C)\), where \(\E^{1}\) is the class of double extensions in \(\C\)~\cite[Lemma 4.3]{EGVdL}.

The claim that \(G_{1}(f)\) is the kernel of \(\rho^{1}_{f}\) amounts to \(G_{1}(f)\) being a normal monomorphism, and this is part of the construction of \(F_{1}\) in~\cite[Subsection 4.2]{EGVdL}, where the sequence
\[
    0\to G_{1}[f]\xrightarrow{G_{1}(f)} X\xrightarrow{\rho^{1}_{f}} F_{1}[f]\to 0
\]
is shown to be short exact. Indeed, \(G(\R{f})\) is a reflexive relation on \(G(X)\)---jointly monic because by naturality of \(\mu\) the composite \((\mu_{X}\times\mu_{X})\circ(G(\pi_{0}),G(\pi_{1}))=(\pi_{0},\pi_{1})\circ\mu_{\R{f}}\) is a monomorphism, and reflexive since \(G\) takes the diagonal of \(\R{f}\) to a diagonal for it---hence an equivalence relation, \(\C\) being Mal'tsev. Its normalisation \(G(\pi_{1})\circ\ker(G(\pi_{0}))\) is therefore a monomorphism, and so is its composite \(G_{1}(f)\) with \(\mu_{X}\). It is a normal mono, because \(\mu_{\R{f}}\circ\ker(G(\pi_{0}))\) is a normal monomorphism (as an intersection of normal monomorphisms, the left-hand square above being a pullback) while, again by naturality of \(\mu\), the equality
\[
    G_{1}(f)=\pi_{1}\circ\mu_{\R{f}}\circ\ker(G(\pi_{0}))
\]
exhibits the monomorphism \(G_{1}(f)\) as its regular image along \(\pi_{1}\), which is again normal in \(X\).

\begin{remark}\label{Remark Relative Commutator}
    The functors \(G\) and \(G_{1}\) are commutators in disguise. In the varietal context of Convention~\ref{Convention Comonadic Homology}, \(G(X)=\Ker(\eta_{X})\) is the verbal subobject of \(X\) determined by the identities which cut \(\X\) out of \(\C\); for \(\C=\Grp\) and \(\X=\Ab\) it is the derived subgroup \([X,X]\). Correspondingly, \(G_{1}[f]\) is a relative commutator of \(\Ker(f)\) and~\(X\). For groups and abelianisation it is the subgroup \([K,X]\), where \(K=\Ker(f)\): the factorisation above then reads \(X\to X/[K,X]\to Y\), the extension \(f\) is central exactly when \([K,X]\) is trivial, and \(G_{1}[f]\) is the subgroup \([F,R]\) which occupies the denominator of the Hopf formula recalled in the introduction. Such constructions go back to the Baer invariants of~\cite{Frohlich,Lue,Furtado-Coelho} and the commutator of~\cite{Higgins}; the relative commutator theory implemented by \(G_{1}\) is developed in~\cite{EveraertCommutator,EverVdLRCT}, with the associated Baer invariants in~\cite{EverVdL1,EverVdL2}.
\end{remark}

\(\Ext(\C)\) is not closed under limits in \(\Arr(\C)\) (Example~\ref{Example Limits Ext}). It is, however, closed under all colimits which exist in \(\Arr(\C)\)---and those are computed pointwise.

\begin{proposition}\label{Proposition Colimits Ext}
    Let \(K\colon{\D \to \Ext(\C)}\) be a diagram on a small category \(\D\). Write \(Kd\colon{Ed\to Bd}\), with \(E\) and \(B\) the induced diagrams in \(\C\). If the colimits of \(E\) and \(B\) exist in \(\C\), then the induced comparison morphism \(k\colon{\colim E\to \colim B}\) is an extension, and together with the pointwise colimit cocone it is the colimit of \(K\) in~\(\Ext(\C)\).
\end{proposition}
\begin{proof}
    Colimits in \(\Arr(\C)\) are computed pointwise, so the colimit of \(K\) in \(\Arr(\C)\) is the arrow \(k\), together with the cocone whose component at \(d\) is the square \((\lambda_{d},\mu_{d})\colon{Kd\to k}\) formed by the colimit cocones
    \[
        \lambda_{d}\colon{Ed\to \colim E}
        \qquad\text{and}\qquad
        \mu_{d}\colon{Bd\to \colim B}.
    \]
    Since \(\Ext(\C)\) is a full subcategory of \(\Arr(\C)\), it suffices to prove that \(k\) is an extension.

    Let \(k=m\circ e\) be the image factorisation of \(k\), with \(e\colon{\colim E\to I}\) a regular epimorphism and \(m\colon{I\to\colim B}\) a monomorphism. For every \(d\) in \(\D\) we have \(\mu_{d}\circ Kd=k\circ \lambda_{d}=m\circ (e\circ\lambda_{d})\), so that, \(Kd\) being a regular and hence a strong epimorphism, the morphism \(\mu_{d}\) factors through \(m\): there is a unique \(\nu_{d}\colon{Bd\to I}\) such that \(\mu_{d}=m\circ\nu_{d}\). By this uniqueness and since \(m\) is a monomorphism, the \(\nu_{d}\) form a cocone on \(B\), which induces a morphism \(t\colon{\colim B\to I}\) satisfying \(t\circ \mu_{d}=\nu_{d}\) for all \(d\). Then \(m\circ t\circ\mu_{d}=\mu_{d}\) for all \(d\), so that \(m\circ t=1_{\colim B}\), the \(\mu_{d}\) being jointly epimorphic. Thus the monomorphism \(m\) is also a split epimorphism, hence an isomorphism, and \(k\) is an extension.
\end{proof}

Of the limits in \(\Ext(\C)\), we shall only need kernel pairs of double extensions. Recall~\cite[Proposition~3.3]{Bourn2003} that in a regular Mal'tsev category---so in particular in any homological category---a commutative square of extensions as in Definition~\ref{Definition Double Extension} is a double extension if and only if the induced morphism \(\R{g}\to\R{g'}\) between the kernel pairs of its horizontal arrows is an extension.

\begin{lemma}\label{Lemma Kernel Pairs Ext}
    Let \((g,g')\colon{a\to b}\) be a morphism in \(\Ext(\C)\) which, as a square in~\(\C\), is a double extension. Then the kernel pair of \((g,g')\) in \(\Arr(\C)\) (which is the arrow \({\R{g}\to \R{g'}}\) induced by \(a\) and \(b\), together with the pointwise projections) lies in~\(\Ext(\C)\), and is the kernel pair of \((g,g')\) in \(\Ext(\C)\).
\end{lemma}
\begin{proof}
    Kernel pairs in \(\Arr(\C)\) are computed pointwise, and by the characterisation of double extensions recalled above, the induced arrow \(\R{g}\to\R{g'}\) is an extension. A limit, taken in \(\Arr(\C)\), of a diagram in the full subcategory \(\Ext(\C)\) which happens to lie in \(\Ext(\C)\) is a limit in \(\Ext(\C)\); the claim follows.
\end{proof}

\begin{example}\label{Example Limits Ext}
    \(\Ext(\C)\) need not be closed in \(\Arr(\C)\) under limits, not even under finite ones. Consider, in the category \(\Ab\) of abelian groups, the quotient map \(q\colon{\mathbb{Z}\to\mathbb{Z}/2}\) as an object of \(\Ext(\Ab)\), and the two endomorphisms \((1_{\mathbb{Z}},1_{\mathbb{Z}/2})\) and \((3\cdot(-),1_{\mathbb{Z}/2})\) of \(q\). Their equaliser in \(\Arr(\Ab)\), computed pointwise, is the arrow \(0\to \mathbb{Z}/2\), which is not an extension. The equaliser does exist in \(\Ext(\Ab)\), namely \(0\to 0\); the point is that the inclusion \(\Ext(\Ab)\to\Arr(\Ab)\) does not preserve it.
\end{example}

\begin{lemma}\label{Lemma Kernels Double Extension}
    Consider a double extension as in Definition~\ref{Definition Double Extension}. The induced morphisms \(\Ker(f)\to\Ker(f')\) and \(\Ker(g)\to\Ker(g')\) are extensions.
\end{lemma}
\begin{proof}
    Write \(c\colon{X\to Z\times_{W}Y}\) for the comparison morphism to the pullback of \(f'\) and \(g'\), which is an extension by assumption. The first projection
    \[
        {g'}^{*}(f')\colon{Z\times_{W}Y\to Z,}
    \]
    the pullback of \(f'\) along \(g'\), has kernel canonically isomorphic to \(\Ker(f')\); under this isomorphism, the induced morphism \(\Ker(f)\to \Ker(f')\) becomes the restriction \(\Ker(f)\to \Ker({g'}^{*}(f'))\) of \(c\). Since \({g'}^{*}(f')\circ c=f\), this restriction is a pullback of~\(c\) along the kernel \(\ker({g'}^{*}(f'))\); extensions being pullback-stable, it is an extension.
\end{proof}

\begin{convention}\label{Convention Generalised Elements}
    Several arguments below are carried out with \defn{generalised elements}: a generalised element of an object \(A\) is a morphism \(x\colon{T\to A}\), and we write \(\phi(x)\) for the composite \(\phi\circ x\), and \(0\) for a zero morphism. As \(\C\) is pointed, each representable functor \(\C(T,-)\) takes values in pointed sets and preserves finite limits. Consequently, any property of a diagram built from finite limits and zero morphisms---for instance that a morphism vanishes, or that a family of morphisms is jointly monic---holds as soon as it holds for all generalised elements.
\end{convention}

\section{The fundamental group}\label{section fundamental group}

Let \(\C\) be a semi-abelian category. Let \(F\colon \C\to \X\) be a Birkhoff reflector and \(F_{1}\colon{\Ext(\C)\to \CExt(\C)}\) the induced centralisation functor.

\begin{lemma}\label{Lemma F-Inverted Extensions}
    An extension \(e\colon{A\to B}\) in \(\C\) is inverted by \(F\)---that is, \(F(e)\) is an isomorphism---if and only if \(\Ker(e)\leq G(A)\).
\end{lemma}
\begin{proof}
    The strong \(\E\)-Birkhoff square of \(e\) is a double extension. Hence Lemma~\ref{Lemma Kernels Double Extension} implies that the restriction \(\Ker(e)\to \Ker(F(e))\) of \(\eta_{A}\) is an extension. If
    \[
        \Ker(e)\leq G(A)=\Ker(\eta_{A}),
    \]
    then this restriction is zero, so that \(\Ker(F(e))=0\) and the regular epimorphism \(F(e)\) is an isomorphism. Conversely, if \(F(e)\) is an isomorphism, then \(\eta_{A}\circ \ker(e)\), which factors through \(\ker(F(e))=0\), is zero, so that \(\Ker(e)\leq \Ker(\eta_{A})=G(A)\).
\end{proof}

\begin{lemma}[Lemma~6.2 in~\cite{EGVdL}]\label{Lemma reflection of centralisation}
    For any extension \(f\colon{X\to Y}\) in \(\C\) we have that \(F(F_{1}(f))\cong F(f)\); more precisely, the functor \(F\) sends the canonical comparison \(\rho^{1}_{f}\colon X\to F_{1}[f]\) to an isomorphism \(F(\rho^{1}_{f})\colon F(X)\to F(F_{1}[f])\) such that
    \[
        F(F_{1}(f)) \circ F(\rho^{1}_{f})=F(f).
    \]
\end{lemma}
\begin{proof}
    By the construction of \(F_{1}\), the kernel \(\Ker(\rho^{1}_{f})=G_{1}[f]\) lies in \(G(X)\), since \(G_{1}(f)\) factors through \(\mu_{X}=\ker(\eta_{X})\). Hence \(F(\rho^{1}_{f})\) is an isomorphism (Lemma~\ref{Lemma F-Inverted Extensions}), and applying \(F\) to \(F_{1}(f)\circ \rho^{1}_{f}=f\) yields \(F(F_{1}(f))\circ F(\rho^{1}_{f})=F(f)\).
\end{proof}

\begin{lemma}\label{Lemma Kernel Reflection Centralisation Comparison}
    If \(f\colon{X\to Y}\) is a split epimorphism, then
    \[
        \Ker(F(f))\cong \Ker(F_{1}(f)).
    \]
\end{lemma}
\begin{proof}
    The central extension \(F_{1}(f)\), being a split epimorphism, is a trivial extension by Lemma~\ref{split central extension}; hence its kernel coincides with the kernel of~\(F(F_{1}(f))\), which is also the kernel of \(F(f)\) by Lemma~\ref{Lemma reflection of centralisation}.
\end{proof}

\begin{lemma}\label{Lemma Kernel Extension and Kernel Pair Centralisation Comparison}
    If \(f\colon{X\to Y}\) is an extension with kernel pair \((\R{f},\pi_{0},\pi_{1})\), then we have an isomorphism \(\Ker(F_{1}(\pi_{0}))\cong \Ker(F_{1}(f))\). As a consequence, \(F_{1}(\pi_{0})\) is a pullback of \(F_{1}(f)\) along \(f\).
\end{lemma}
\begin{proof}
    Consider the pullback on the left, viewed as a morphism of extensions from~\(\pi_{0}\) to~\(f\), and its image through the functor \(F_{1}\) on the right.
    \[
        \xymatrix{\R{f} \pullback \ar[d]_-{\pi_{0}} \ar[r]^-{\pi_{1}} & X \ar[d]^-{f} \\
        X \ar[r]_-{f} & Y}
        \qquad \qquad
        \xymatrix{F_{1}[\pi_{0}] \ar[d]_-{F_{1}(\pi_{0})} \ar[r] & F_{1}[f] \ar[d]^-{F_{1}(f)} \\
        X \ar[r]_-{f} & Y}
    \]
    The latter square need no longer be a pullback, but it is still a double extension since \(F_{1}\) preserves double extensions. The comparison
    \[
        \xymatrix{\R{f} \ar[d]_-{\rho^{1}_{\pi_{0}}} \ar[r]^-{\pi_{1}} & X \ar[d]^-{\rho^{1}_{f}} \\
        F_{1}[\pi_{0}] \ar[r] & F_{1}[f] }
    \]
    between them is in fact a pullback. Indeed, the induced morphism
    \[
        c\colon{\R{f}\to X\times_{F_{1}[f]}F_{1}[\pi_{0}]}
    \]
    to the pullback is an extension. To see this, note first that, viewed as a morphism \({\pi_{0}\to f}\) in \(\Ext(\C)\), the square \((\pi_{1},f)\) is a double extension (being a pullback of extensions, so that its comparison morphism is an isomorphism) and hence lies in \(\E^{1}\). Since \(\CExt(\C)\) is a strongly \(\E^{1}\)-Birkhoff subcategory of \(\Ext(\C)\), the square formed by \((\pi_{1},f)\), its image under \(F_{1}\) and the units of the adjunction \(F_{1}\dashv{\subseteq_{1}}\) is a double extension in \(\Ext(\C)\); in particular, its comparison morphism to the pullback of \(F_{1}(\pi_{0})\) and \(f\) over \(F_{1}(f)\) lies in \(\E^{1}\), so that it is a double extension in \(\C\), all four edges of which are extensions. That pullback is computed pointwise~\cite[Proposition 3.5]{EGVdL}, whence the comparison morphism in question is a square in \(\C\) whose domain edge is \(c\). The kernel of \(c\) is \(\Ker(\pi_{1})\cap \Ker(\rho^{1}_{\pi_{0}})\), the pullback projections being jointly monic. Applying the construction of the centralisation functor in Section~\ref{section basic concepts} to the extension \(\pi_{0}\), we find \(\Ker(\rho^{1}_{\pi_{0}})=G_{1}[\pi_{0}]\leq \Ker(\pi_{0})\), so that
    \[
        \Ker(c)=\Ker(\pi_{1})\cap \Ker(\rho^{1}_{\pi_{0}})\leq \Ker(\pi_{1})\cap\Ker(\pi_{0})=0\text{,}
    \]
    the kernel pair projections \(\pi_{0}\) and \(\pi_{1}\) being jointly monic. An extension with trivial kernel being an isomorphism, the morphism \(c\) is invertible, and the comparison square is a pullback. As a consequence---kernels being computed by pullbacks, and limits commuting with limits---the square
    \[
        \xymatrix{\Ker(\pi_{0}) \pullback \ar[d] \ar@{=}[r] & \Ker(f) \ar[d] \\
            \Ker(F_{1}(\pi_{0})) \ar[r] & \Ker(F_{1}(f)) }
    \]
    is also a pullback. The map on the right is a regular epimorphism: the unit square \((\rho^{1}_{f},1_{Y})\colon{f\to F_{1}(f)}\) is a double extension---its comparison morphism to the pullback being \(\rho^{1}_{f}\) itself---so that Lemma~\ref{Lemma Kernels Double Extension} applies. Hence the bottom map \(\Ker(F_{1}(\pi_{0}))\to \Ker(F_{1}(f))\) is a regular epimorphism too. Since pullbacks along regular epimorphisms reflect monomorphisms, it is also a monomorphism. This proves the first claim. For the second, the induced comparison from \(F_{1}[\pi_{0}]\) to the pullback of \(F_{1}(f)\) along \(f\) is a morphism of extensions over \(X\) which restricts to this isomorphism between the kernels; by the Short Five Lemma it is invertible, so that \(F_{1}(\pi_{0})\) is indeed a pullback of \(F_{1}(f)\) along \(f\).
\end{proof}

By symmetry, \(F_{1}(\pi_{1})\) is likewise a pullback of \(F_{1}(f)\) along \(f\).

\begin{lemma}\label{Lemma Simple Galois Groupoid}
    For any extension \(f\colon{X\to Y}\) we have
    \[
        F(\R{F_{1}(f)})\cong F(\R{f})
    \]
    as groupoids in \(\X\).
\end{lemma}
\begin{proof}
    Write \(\pi'_{0}\), \(\pi'_{1}\colon{\R{F_{1}(f)}\to F_{1}[f]}\) for the kernel pair projections of the centralisation \(F_{1}(f)\), and \(r=\rho^{1}_{f}\times_{Y}\rho^{1}_{f}\colon{\R{f}\to \R{F_{1}(f)}}\) for the comparison induced by \(\rho^{1}_{f}\). Reflecting into \(\X\) and taking kernels of \(F(\pi_{0})\) and \(F(\pi'_{0})\) to the left, we obtain the diagram
    \[
        \xymatrix{\Ker(F(\pi_{0})) \ar[d]_-{k} \ar[r] & F(\R{f}) \ar@<.5ex>[r]^-{F(\pi_{0})} \ar[d]_-{F(r)} \ar@<-.5ex>[r]_-{F(\pi_{1})} & F(X) \ar[d]_-{\cong}^-{F(\rho^{1}_{f})} \\
        \Ker(F(\pi'_{0})) \ar[r] & F(\R{F_{1}(f)}) \ar@<.5ex>[r]^-{F(\pi'_{0})} \ar@<-.5ex>[r]_-{F(\pi'_{1})} & F(F_{1}[f]). }
    \]
    From Lemma~\ref{Lemma reflection of centralisation} we already know that \(F(\rho_{f}^{1})\) is an isomorphism.

    Lemma~\ref{Lemma Kernel Reflection Centralisation Comparison} tells us that the kernel of \(F(\pi_{0})\)---the projection \(\pi_{0}\) being split by the diagonal of \(\R{f}\)---may be obtained as the kernel of the centralisation \(F_{1}(\pi_{0})\). From Lemma~\ref{Lemma Kernel Extension and Kernel Pair Centralisation Comparison}, in turn, we know that this kernel is precisely the kernel of \(F_{1}(f)\). This, however, is also the kernel of \(\pi'_{0}\), which coincides with the kernel of \(F(\pi'_{0})\) since \(F_{1}(f)\) is a central extension, so that \(\pi'_{0}\) is a trivial extension.

    Both kernels in the diagram are thus copies of \(\Ker(F_{1}(f))\); it remains to be checked that the comparison \(k\) between them is itself invertible. The strong \(\E\)-Birkhoff squares of \(\pi_{0}\) and of \(\pi'_{0}\) are double extensions, so that by Lemma~\ref{Lemma Kernels Double Extension} the units restrict to extensions
    \[
        \alpha\colon{\Ker(\pi_{0})\to \Ker(F(\pi_{0}))}\qquad\text{and}\qquad \beta\colon{\Ker(\pi'_{0})\to \Ker(F(\pi'_{0}))}\text{,}
    \]
    the latter an isomorphism because \(\pi'_{0}\) is a trivial extension, so that the naturality square of \(\eta\) at \(\pi'_{0}\) is a pullback. Writing \(r'\colon{\Ker(\pi_{0})\to \Ker(\pi'_{0})}\) for the restriction of \(r\), the naturality of \(\eta\) gives \(k\circ \alpha=\beta\circ r'\). We now identify \(\Ker(\pi_{0})\) with \(\Ker(f)\) along \(\pi_{1}\). On the one hand, \(\alpha\) being a restriction of \(\eta_{\R{f}}\),
    \[
        \Ker(\alpha)=\Ker(\pi_{0})\cap G(\R{f})=\Ker(G(\pi_{0}))=G_{1}[f]
    \]
    by the construction of the centralisation functor, the left-hand square in the diagram defining \(G_{1}[f]\) being a pullback. On the other hand, \(r'\) becomes the restriction \(\Ker(f)\to \Ker(F_{1}(f))\) of \(\rho^{1}_{f}\), which is a regular epimorphism by Lemma~\ref{Lemma Kernels Double Extension} applied to the unit square \((\rho^{1}_{f},1_{Y})\colon{f\to F_{1}(f)}\), and whose kernel is \(\Ker(f)\cap\Ker(\rho^{1}_{f})=G_{1}[f]\) as well. Hence \(\alpha\) and \(\beta\circ r'\) are regular epimorphisms with the same kernel: each is a cokernel of the inclusion \(G_{1}[f]\leq \Ker(f)\), and the induced comparison \(k\) is an isomorphism.

    Since \(F\) preserves extensions, both rows are short exact sequences, and the Short Five Lemma now tells us that also \(F(r)\) is an isomorphism.
\end{proof}

\begin{remark}\label{Remark No Centralisation}
    Lemma~\ref{Lemma Simple Galois Groupoid} says that \(F\) does not distinguish the kernel pair of an extension from that of its centralisation: there is no need to centralise in order to obtain the Galois groupoid, and hence the fundamental group. This is in essence Theorem~2.17 of~\cite{Bourn-Rodelo2}, proved there by means of the comprehensive factorization.
\end{remark}

\begin{definition}
    An object \(P\) in \(\C\) is \defn{projective} if for every regular epimorphism \(f\colon{X\to Y}\) and every morphism \(p\colon{P\to Y}\) there exists a morphism \(p'\colon{P\to X}\) such that \(f\circ p'=p\). \(\C\) is said to have \defn{enough projectives} if for every object \(Y\) there exists a projective object \(P\) and a regular epimorphism \(P\to Y\).
\end{definition}

Suppose \(\C\) has enough projectives. A \defn{weakly universal central extension} of an object \(Y\) in \(\C\) is a central extension \(f\colon{X\to Y}\) such that for every central extension \(f'\colon{X'\to Y}\) there exists a morphism \(g\colon{X\to X'}\) with \(f'\circ g=f\). Given any extension \(f\colon{X\to Y}\), its \defn{Galois groupoid} is the internal groupoid
\[
    F(\R{f})\two F(X)
\]
the image under \(F\) of the kernel pair \((\pi_{0},\pi_{1})\) of \(f\)---a groupoid indeed, by Proposition~\ref{Birkhoff Reflector Nerve} below; its \defn{Galois group} is the intersection \(\Ker(F(\pi_{0}))\cap \Ker(F(\pi_{1}))\) of the kernels of the two projections of this groupoid (cf.~\cite{Janelidze:Hopf})---equivalently, the kernel of \((F(\pi_{0}),F(\pi_{1}))\colon{F(\R{f})\to F(X)\times F(X)}\), which measures the failure of~\(F\) to preserve the kernel pair of \(f\). The \defn{fundamental group} \(\pi_{1}^{F}(Y)\) of \(Y\) is the Galois group of a weakly universal central extension of \(Y\). This does not depend on the chosen weakly universal central extension: any two of them admit comparison morphisms over \(Y\) in both directions, which exhibit their kernel pairs as equivalent internal groupoids and so induce an isomorphism between the corresponding Galois groups~\cite{Janelidze:Hopf}. Finally, when \(f\colon{P\to Y}\) is an extension with projective domain, its centralisation \(F_{1}(f)\) is a weakly universal central extension of~\(Y\) \cite[Theorem~5.1]{Janelidze:Hopf}; consequently \(\pi_{1}^{F}(Y)\) may be computed as the intersection of the kernels of the two projections of the Galois groupoid \(F(\R{F_{1}(f)})\two F(F_{1}[f])\).

\begin{theorem}[cf.~{\cite[Theorem~2.17]{Bourn-Rodelo2}}]\label{Theorem Interpretation Fundamental Group}
    For any object \(Y\) in a semi-abelian category with enough projectives~\(\C\), its fundamental group \(\pi_{1}^{F}(Y)\) with respect to a Birkhoff reflector \(F\colon \C\to \X\) may be computed as the intersection of the kernels of the two projections in the groupoid \(F(\R{f})\two F(X)\), where \(f\colon{X\to Y}\) is a projective presentation of \(Y\).
\end{theorem}

We include a proof nonetheless: it is this explicit, centralisation-free form that the induction of Section~\ref{section induction} carries up the tower of Galois structures.

\begin{proof}
    As observed just above, \(\pi_{1}^{F}(Y)\) is the intersection of the kernels of the two projections of the Galois groupoid \(F(\R{F_{1}(f)})\two F(F_{1}[f])\) of the centralisation of~\(f\). The comparison \((r,\rho^{1}_{f})\colon{(\R{f}\two X)\to(\R{F_{1}(f)}\two F_{1}[f])}\) is inverted by \(F\) in both components: \(r\) by Lemma~\ref{Lemma Simple Galois Groupoid}, and \(\rho^{1}_{f}\) by Lemma~\ref{Lemma reflection of centralisation}. It thus induces an isomorphism of internal groupoids
    \[
        F(\R{f})\two F(X)\quad\cong\quad F(\R{F_{1}(f)})\two F(F_{1}[f])
    \]
    under which \(\pi_{1}^{F}(Y)\) is identified with the intersection of the kernels of the two projections of \(F(\R{f})\two F(X)\).
\end{proof}

\section{The reflection of a nerve}\label{section reflection nerve}

The nerve \(\Nerve(C)\) of an internal category \(C=(C_{1},C_{0},d_1,d_0,s_0)\)
\[
    \xymatrix{C_{1}\ar@<1ex>[r]^-{d_1} \ar@<-1ex>[r]_-{d_0} & C_{0} \ar[l]|-{s_0}}
\]
is computed by successive pullbacks of the split epimorphisms \(d_1\) and \(d_0\). It is the simplicial object defined by \(\Nerve(C)_{0}=C_{0}\), \(\Nerve(C)_{1}=C_{1}\), \(\Nerve(C)_{2}=C_{1}\times_{(d_0,d_1)}C_{1}\) and
\[
    \Nerve(C)_{n}=\underbrace{C_{1}\times_{(d_0,d_1)}C_{1}\times_{(d_0,d_1)}\cdots\times_{(d_0,d_1)}C_{1}}_{\text{\(n\) factors}}
\]
for \(n\geq 3\), with the relevant face and degeneracy operators. Throughout, we write \(\del_{i}\) and \(s_{i}\) for the faces and degeneracies of a simplicial object, keeping the letters \(d_{0}\), \(d_{1}\) and \(s_{0}\) for the structure of a reflexive graph; in degree one the two agree, so that \(\del_{0}=d_{0}\) and \(\del_{1}=d_{1}\). More explicitly, we can define \(\Nerve(C)\) via the sequence of pullbacks
\[
    \xymatrix{\cdots \ar[r] & C_3 \pullback \ar[r]^-{\del_0} \ar[d]_-{\del_3} & C_2 \pullback \ar[r]^-{\del_0} \ar[d]_-{\del_2} & C_1 \ar[d]_-{\del_1}\\
    \cdots \ar[r] & C_2 \ar[r]^-{\del_0} & C_1 \ar[r]^-{\del_0} & C_0}
\]
where \(C_n\coloneq \Nerve(C)_{n}\). Then \(\del_1\colon C_2\to C_1\) is the composition in \(C\), and for every \(n\geq 3\) and every \(0<i<n\) the face operator \(\del_i\colon C_n\to C_{n-1}\) is defined via the universal property of the pullback as the unique morphism such that \(\del_0\del_i=\del_{i-1}\del_0\) and \(\del_{n-1}\del_i=\del_i\del_n\); the degeneracy operators are defined similarly for every \(n\geq 2\) and every \(0\leq i<n\) as the unique morphism \(s_i\colon C_{n-1}\to C_n\) such that \(\del_0s_i=s_{i-1}\del_0\) and \(\del_ns_i=s_i\del_{n-1}\), where for \(i=0\) the first equation is replaced by \(\del_0s_0=1_{C_{n-1}}\), and for \(i=n-1\) the second by \(\del_ns_{n-1}=1_{C_{n-1}}\). This defines a functor \(\Nerve\colon{\Cat(\C)\to \s(\C)}\) on the category \(\Cat(\C)\) of internal categories in \(\C\). Since \(\C\) is a Mal'tsev category, every internal category in \(\C\) is automatically an internal groupoid~\cite{Carboni-Pedicchio-Pirovano}, so that \(\Cat(\C)=\Grpd(\C)\), the category of internal groupoids in \(\C\); we use the latter notation from now on. The image of the nerve functor can be characterised as follows.

Recall that the \defn{Moore complex} of a simplicial object \(S\) in a pointed category with finite limits is the chain complex \(N(S)\) with \(N_{0}S=S_{0}\),
\[
    N_{n}S=\bigcap_{i=0}^{n-1}\Ker(\del_{i}\colon S_{n}\to S_{n-1})
\]
and differential the morphism \(N_{n}S\to N_{n-1}S\) induced by \(\del_{n}\). When the surrounding category is homological, the \defn{homology objects} \(H_{n}S\) of \(S\) are those of its Moore complex; see~\cite{VdLinden:Doc, EverVdL2} for details.

\begin{definition}[\cite{Tierney-Vogel}]\label{Definition Simplicial Kernel}
    A \defn{simplicial kernel} of a sequence of morphisms \(f_0\), \dots, \(f_n\colon X\to Y\) is a sequence of morphisms \(k_0,\ldots,k_{n+1}\colon K\to X\) such that \(f_ik_j=f_{j-1}k_i\) for all \(0\leq i<j\leq n+1\), and which is universal with respect to this property.
\end{definition}

Given a simplicial object \(S\) and \(n\geq 1\), we write \(K_{n}S\) for the simplicial kernel of the sequence of faces \(\del_{0},\dots,\del_{n-1}\colon{S_{n-1}\to S_{n-2}}\) (where for \(n=1\) we take the empty sequence, so that \(K_{1}S=S_{0}\times S_{0}\)) and \(\sigma_{n}\colon{S_{n}\to K_{n}S}\) for the induced comparison morphism, whose components are the faces of \(S\). Dropping one of the universal morphisms \(k_{0}\), \dots, \(k_{n}\) of \(K_{n}S\), say \(k_{l}\), together with all relations involving it, yields the \defn{horn object} \(\Lambda^{n}_{l}S\). The simplicial object \(S\) is \defn{Kan} when for each \(n\) and \(l\) the induced comparison \(S_{n}\to \Lambda^{n}_{l}S\) is a regular epimorphism. In a regular Mal'tsev category, every simplicial object is Kan~\cite{Carboni-Kelly-Pedicchio}.

The next fundamental result is likely known---in any left exact category, Bourn characterises the nerves of \(n\)-groupoids by a simplicial-kernel condition~\cite{Bourn:Nerve}---but we couldn't locate a precise reference for its Moore-complex form, so we decided to include a full proof.

\begin{proposition}[cf.~\cite{Bourn:Dold-Kan,Bourn:Nerve}]\label{Proposition Characterisation Nerve}
    In a homological category, a simplicial object \(S\) is the nerve of an internal groupoid if and only if the Moore complex of \(S\) is trivial in degrees greater than or equal to \(2\).
\end{proposition}
\begin{proof}
    First consider \(S=\Nerve(C)\) for an internal groupoid \(C\), and take \(n\geq 2\). We argue with generalised elements: applying a representable functor \(\C(T,-)\), which preserves all the limits involved, reduces the claim to the case of a pointed set-theoretical groupoid. An \(n\)-simplex is then a string \(x\coloneq(f_{1},\dots,f_{n})\) of composable arrows, with \(\del_{0}(f_{1},\dots,f_{n})=(f_{2},\dots,f_{n})\) and \(\del_{1}(f_{1},\dots,f_{n})=(f_{2}\circ f_{1},f_{3},\dots,f_{n})\), and the zero simplex is the string each of whose arrows is the identity at the base point. If \(\del_{0}x=0\), then \(f_{2}=\cdots=f_{n}=0\); if moreover \(\del_{1}x=0\), then also \(f_{1}=f_{2}\circ f_{1}=0\). Hence already \(\Ker(\del_{0})\cap\Ker(\del_{1})=0\), so that \(N_{n}S=0\).

    Conversely, suppose that \(N_{n}S=0\) for all \(n\geq 2\). Since \(\C\) is homological and \(\Ker(\del_{0})\cap\Ker(\del_{1})=N_{2}S=0\), the pair \((\del_{0},\del_{1})\colon{S_{2}\to S_{1}}\) is jointly monic; by \(\del_{0}\del_{1}=\del_{0}\del_{0}\) it induces a monomorphism \(\varphi\coloneq(\del_{0},\del_{1})\colon{S_{2}\to \Eq(\del_{0}\colon S_{1}\to S_{0})}\), which is a regular epimorphism by the Kan property, hence an isomorphism.

    Likewise, the comparison \(\psi\coloneq(\del_{0},\del_{2})\colon{S_{2}\to S_{1}\times_{(\del_{1},\del_{0})}S_{1}}\) to the object of composable pairs (that is, \((a_{0},a_{2})\) with \(\del_{1}(a_{0})=\del_{0}(a_{2})\); the comparison is well defined by the simplicial identity \(\del_{0}\del_{2}=\del_{1}\del_{0}\)) is an isomorphism. It is a regular epimorphism by the Kan property. To see that it is a monomorphism, let \(x\) be a generalised element of \(S_{2}\) such that \(\del_{0}x=0=\del_{2}x\). Then
    \[
        \varphi(s_{1}\del_{1}x)=(s_{0}\del_{0}\del_{1}x,\del_{1}x)=(s_{0}\del_{0}\del_{0}x,\del_{1}x)=(0,\del_{1}x)=\varphi(x)\text{,}
    \]
    so that \(x=s_{1}\del_{1}x\) and hence \(\del_{1}x=\del_{2}s_{1}\del_{1}x=\del_{2}x=0\); thus \(x\in N_{2}S=0\).

    Now \(m\coloneq \del_{1}\circ\psi^{-1}\) is a multiplication on the reflexive graph \((S_{1},S_{0},\del_{0},\del_{1},s_{0})\): the computations \(\psi(s_{0}a)=(a,s_{0}\del_{1}a)\) and \(\psi(s_{1}a)=(s_{0}\del_{0}a,a)\) yield the unit axioms \(m(a,s_{0}\del_{1}a)=\del_{1}s_{0}a=a\) and \(m(s_{0}\del_{0}a,a)=\del_{1}s_{1}a=a\). In a Mal'tsev category, every multiplicative reflexive graph is an internal groupoid~\cite{Carboni-Pedicchio-Pirovano}, so \(S_{1}\two S_{0}\) carries a (necessarily unique) internal groupoid structure. It remains to identify \(S\) with the nerve \(T\coloneq\Nerve(S_{1}\two S_{0})\) of this groupoid.

    We first show that for every \(n\geq 3\), the comparison \(\sigma_{n}\colon{S_{n}\to K_{n}S}\) is an isomorphism. Its kernel \(\bigcap_{i=0}^{n}\Ker(\del_{i})\) lies in \(N_{n}S=0\), so it is a monomorphism. Now fix \(0\leq l\leq n\) and consider the factorisation \(S_{n}\to K_{n}S\to \Lambda^{n}_{l}S\) of the comparison to the horn object. The forgetful morphism \(K_{n}S\to\Lambda^{n}_{l}S\) is a monomorphism: by the simplicial kernel relations, every face of the omitted component \(k_{l}\) is determined by the remaining components, and the faces of \(S_{n-1}\) are jointly monic since their joint kernel lies in \(N_{n-1}S=0\). The comparison \(S_{n}\to\Lambda^{n}_{l}S\) being a regular epimorphism by the Kan property, it follows that \(K_{n}S\to \Lambda^{n}_{l}S\) is an isomorphism and that \(\sigma_{n}\) is a regular epimorphism, hence an isomorphism.

    By the first part of this proof, \(T\) also has trivial Moore complex in degrees \(\geq 2\), so the comparisons \(\sigma_{n}\colon{T_{n}\to K_{n}T}\) are isomorphisms for \(n\geq 3\) as well.

    Finally, the \(2\)-truncations of \(S\) and \(T\) are isomorphic: take identities in degrees \(0\) and \(1\), and \(\psi\colon{S_{2}\to T_{2}}\) in degree \(2\). This is compatible with all faces and degeneracies: writing \(\del^{T}_{i}\) for the faces of \(T\), the equalities \(\del^{T}_{0}\psi=\del_{0}\) and \(\del^{T}_{2}\psi=\del_{2}\) hold by the definition of \(\psi\), \(\del^{T}_{1}\psi=m\psi=\del_{1}\) by the definition of \(m\), and the compatibilities with \(s_{0}\) and \(s_{1}\) are precisely the computations \(\psi(s_{0}a)=(a,s_{0}\del_{1}a)\) and \(\psi(s_{1}a)=(s_{0}\del_{0}a,a)\) above. By induction on \(n\geq 3\), this isomorphism of truncations extends degree by degree to an isomorphism \(S\cong T\) of simplicial objects: the isomorphisms below degree \(n\) induce an isomorphism \(K_{n}S\cong K_{n}T\) of simplicial kernels, which composed with the isomorphisms \(\sigma_{n}\) yields an isomorphism \(S_{n}\cong T_{n}\); this commutes with the faces, which are the simplicial kernel components, and then also with the degeneracies, which are uniquely determined by the simplicial identities since the simplicial kernel components are jointly monic.
\end{proof}

Later on we shall need

\begin{lemma}\label{Lemma H_{1} Zero Means Relation}
    If \(C\) is an internal category, then it is an internal equivalence relation if and only if \(H_{1}C=H_{1}\Nerve(C)=0\).
\end{lemma}
\begin{proof}
    The morphism \((d_1,d_0)\colon {C_{1}\to C_{0}\times C_{0}}\) is a monomorphism if and only if its kernel \(H_{1}\Nerve(C)\) is zero.
\end{proof}

\(\Grpd(\C)\) is a full subcategory of the functor category \(\RG(\C)\) of reflexive graphs in \(\C\), closed under subobjects and regular quotients~\cite[Section~3]{Gran:Central-Extensions}; hence limits and regular epimorphisms are computed levelwise in \(\Grpd(\C)\) as they are in \(\RG(\C)\). In particular, a morphism of internal groupoids is a regular epimorphism precisely when both components are extensions in \(\C\).

\begin{lemma}\label{Lemma Nerve exact}
    The functor \(\Nerve\colon{\Grpd(\C)\to \s\C}\) preserves limits and regular epimorphisms; in particular, it is exact.
\end{lemma}
\begin{proof}
    Limits are preserved since each \(\Nerve(-)_{n}\) is built from finite limits in \(\C\), and limits commute with limits. As for regular epimorphisms: let \(g\colon{C\to D}\) be a regular epimorphism of internal groupoids, so that \(g_{1}\) and \(g_{0}\) are extensions. We show by induction that each \(g_{n}\colon{C_{n}\to D_{n}}\) is an extension, there being nothing to prove for \(n\leq 1\). For \(n\geq 2\), the object \(C_{n}\) is the pullback of \(\del_{0}\) along \(\del_{n-1}\colon{C_{n-1}\to C_{n-2}}\), both split by degeneracies, and likewise for \(D_{n}\); in a regular Mal'tsev category, a levelwise extension between such pullbacks induces an extension between the pullbacks~\cite[Proposition~4.1]{Bourn2003}.
\end{proof}

The nerve allows us to speak of the homology of an internal groupoid~\cite{EKVdL}: we set \(H_{n}C\coloneq H_{n}\Nerve(C)\), so that \(H_{1}C=\Ker(d_{0})\cap \Ker(d_{1})\) and, as for any simplicial object~\cite{EverVdL2}, \(H_{0}C=\coeq(d_{0},d_{1})\). These homology objects are preserved by any exact functor:

\begin{lemma}\label{Lemma Exact Functor Groupoid Homology}
    Let \(U\) be an exact functor between semi-abelian categories. Then for any internal groupoid \(C\), the canonical comparisons \(U(H_{1}C)\to H_{1}(UC)\) and \(U(H_{0}C)\to H_{0}(UC)\) are isomorphisms.
\end{lemma}
\begin{proof}
    Since \(U\) preserves finite limits, \(UC\) is again an internal groupoid, and
    \[
        U(H_{1}C)=U(\Ker(d_{0})\cap\Ker(d_{1}))\cong \Ker(U(d_{0}))\cap\Ker(U(d_{1}))=H_{1}(UC)\text{,}
    \]
    intersections of kernels being finite limits.

    As for \(H_{0}\), let \(S\) be the image of \((d_{1},d_{0})\colon{C_{1}\to C_{0}\times C_{0}}\), so that \(\coeq(d_{0},d_{1})=\coeq(S)\). The relation \(S\) is reflexive, hence an equivalence relation since we are in a Mal'tsev category~\cite{Carboni-Kelly-Pedicchio}, and effective by exactness of the surrounding category: \(S\) is the kernel pair of its own coequaliser \(q\). Now \(U\) preserves the image factorisation of \((d_{1},d_{0})\), so that \(U(S)\) is the image of \((U(d_{1}),U(d_{0}))\); it also preserves the extension \(q\) and its kernel pair, so that \(U(q)\) is an extension with kernel pair \(U(S)\), hence the coequaliser of \(U(S)\). Combining, we find
    \[
        U(H_{0}C)=U(\coeq(S))\cong \coeq(U(S))=\coeq(U(d_{0}),U(d_{1}))=H_{0}(UC)\text{.}\qedhere
    \]
\end{proof}

From~\cite[Lemma~4.4]{EGVdL} we know that

\begin{lemma}\label{Lemma Pullback Preservation Split}
    The Birkhoff reflector \(F\colon \C\to \X\) preserves pullbacks of split epimorphisms along split epimorphisms.\noproof
\end{lemma}

\begin{proposition}[Birkhoff reflectors commute with the nerve construction]\label{Birkhoff Reflector Nerve}
    If \(F\colon{\C\to \X}\) is a Birkhoff reflector of a semi-abelian category \(\C\), then for any internal groupoid \(C\) in \(\C\), the reflexive graph \(F(C)\), obtained by applying \(F\) degree-wise, is again an internal groupoid, and
    \[
        F(\Nerve(C))\cong \Nerve(F(C)).
    \]
\end{proposition}
\begin{proof}
    The nerve \(\Nerve(C)\) is built degree-wise from pullbacks of the split epimorphisms \(d_{0}\) and \(d_{1}\) along one another, with the degeneracies as sections. By Lemma~\ref{Lemma Pullback Preservation Split} the reflector \(F\) preserves these pullbacks, so that \(F(\Nerve(C))\) agrees degree-wise with the nerve of the reflexive graph \(F(C)\); this is the isomorphism \(F(\Nerve(C))\cong\Nerve(F(C))\). In particular, \(F(C)\) is a multiplicative reflexive graph: its composition is the image
    \[
        F(C_{1})\times_{F(C_{0})}F(C_{1})\cong F(C_{2})\xrightarrow{F(\del_{1})}F(C_{1})
    \]
    of that of \(C\). Hence it is an internal groupoid, \(\X\) being a Mal'tsev category~\cite{Carboni-Pedicchio-Pirovano}.
\end{proof}

In particular, we use this on the \Tsjech\ complex of a map. Given a morphism \(f\colon{X\to Y}\), we may take its kernel pair, and then complete this reflexive graph to a simplicial object by taking successive simplicial kernels in the sense of Definition~\ref{Definition Simplicial Kernel}. This yields a simplicial object \(\Cech(f)\), the \defn{\Tsjech\ complex} of \(f\). In case \(f\colon X\to Y\) is a projective presentation, this simplicial object is called a \defn{\Tsjech\ resolution of~\(Y\)}.

Explicitly, write \(X^{n}_{Y}\) for the \(n\)-fold fibre power \(X\times_{Y}\cdots\times_{Y}X\) of \(f\), the wide pullback of \(n\) copies of \(f\) along itself. Then \(\Cech(f)_{n}=X^{n+1}_{Y}\), the face \(\del_{i}\colon{X^{n+1}_{Y}\to X^{n}_{Y}}\) omits the \(i\)-th coordinate, and the degeneracy \(s_{i}\colon{X^{n}_{Y}\to X^{n+1}_{Y}}\) repeats it. Indeed, fibre powers and simplicial kernels being limits, this may be checked in \(\Set\).

\begin{lemma}\label{Cech Is Nerve}
    For any morphism \(f\colon{X\to Y}\), we have \(\Cech(f)\cong\Nerve(\R{f})\).
\end{lemma}
\begin{proof}
    In degree \(n\), the nerve of the equivalence relation \((\R{f},\pi_{0},\pi_{1})\) is the limit of strings of \(n\) composable pairs, and a pointwise verification identifies it with \(X^{n+1}_{Y}\): a string \(((a_{1},a_{0}),(a_{2},a_{1}),\dots,(a_{n},a_{n-1}))\) corresponds to the tuple \((a_{0},\dots,a_{n})\). Under this identification and the explicit description of \(\Cech(f)\) given above, the faces and degeneracies of the two simplicial objects agree.
\end{proof}

The \defn{first \Tsjech\ derived functor} \(\HH_{1}F\) of a functor \(F\colon{\C\to \X}\) is defined on an object \(Y\) as the first homology of the simplicial object \(F(\Cech(f))\), where \(f\colon{X\to Y}\) is a projective presentation of \(Y\). In degree zero nothing new arises: when \(F\) is our Birkhoff reflector, \(H_{0}F(\Cech(f))=\coeq(F(\pi_{0}),F(\pi_{1}))\cong F(Y)\), the left adjoint \(F\) preserving the coequaliser \(f\) of the pair \((\pi_{0},\pi_{1})\). Thus \(\HH_{0}F=F\), which by Convention~\ref{Convention Comonadic Homology} below is \(H_{1}({-},F)\), so that the comparison of Theorem~\ref{Theorem Comparison Level n} extends to degree zero.

\begin{proposition}\label{Cech on Split Epi}
    If \(f\colon{X\to Y}\) is a split epimorphism, then \(H_{n}F(\Cech(f))=0\) for all \(n\geq 1\). In particular, if \(Y\) is projective, then \(\HH_{1}F(Y)=0\).
\end{proposition}
\begin{proof}
    By Lemma~\ref{Lemma Pullback Preservation Split}, the Birkhoff reflector \(F\) preserves the kernel pair of~\(f\), so that \(F(\R{f})=\R{F(f)}\) is again a kernel pair, and in particular an equivalence relation. By Lemma~\ref{Cech Is Nerve} and Proposition~\ref{Birkhoff Reflector Nerve},
    \[
        F(\Cech(f))\cong F(\Nerve(\R{f}))\cong\Nerve(F(\R{f}))\text{.}
    \]
    Being the nerve of a groupoid, this simplicial object has a Moore complex which is trivial in degrees \(\geq 2\) by Proposition~\ref{Proposition Characterisation Nerve}, so that \(H_{n}F(\Cech(f))=0\) for all \(n\geq 2\). Moreover, \(H_{1}F(\Cech(f))=\Ker(F(\pi_{0}))\cap \Ker(F(\pi_{1}))=0\), the pair \((F(\pi_{0}),F(\pi_{1}))\) being jointly monic as the kernel pair of \(F(f)\). For the second statement, note that any projective presentation of a projective object \(Y\) is a split epimorphism, so that \(\HH_{1}F(Y)=H_{1}F(\Cech(f))=0\).
\end{proof}

\begin{proposition}\label{Proposition HH_1 as Homology of a Groupoid}
    For any projective presentation \(f\colon{X\to Y}\) of an object \(Y\) we have that \(\HH_{1}F(Y)\) is the first homology of the groupoid \(F(\R{f})\two F(X)\).
\end{proposition}
\begin{proof}
    In the proof of Proposition~\ref{Cech on Split Epi}, the isomorphism \(F(\Cech(f))\cong\Nerve(F(\R{f}))\) was obtained for an arbitrary morphism \(f\); the splitness of \(f\) entered only afterwards, to make \(F(\R{f})\) an equivalence relation. Hence \(\HH_{1}F(Y)\), the first homology of \(F(\Cech(f))\), is the first homology of a nerve, whose Moore complex is trivial in degrees greater than or equal to \(2\). It is therefore the intersection of the kernels of the two faces of the reflexive graph \(F(\R{f})\two F(X)\), which is the first homology of that groupoid.
\end{proof}

In particular \(\HH_{1}F\colon \C\to \X\) is well defined: it does not depend on the chosen projective presentation of \(Y\), as also observed in~\cite{Culot-Everaert}. Recall that two internal groupoids \(C\) and \(D\) are \defn{homotopy equivalent} when they are equivalent as internal categories: there are internal functors \(u\colon{C\to D}\) and \(v\colon{D\to C}\) together with internal natural transformations \(vu\Rightarrow 1_{C}\) and \(uv\Rightarrow 1_{D}\)~\cite{EKVdL}. Under the nerve, an internal natural transformation becomes a simplicial homotopy~\cite{EKVdL}, and simplicially homotopic morphisms induce the same morphism in homology~\cite{EverVdL2}; hence homotopy equivalent internal groupoids have isomorphic homology objects. Since the Birkhoff reflector~\(F\) commutes with the nerve of the kernel-pair groupoids at hand (Proposition~\ref{Birkhoff Reflector Nerve}), it carries such a homotopy equivalence to a homotopy equivalence in \(\X\); this is the \(F\)-invariance of \(H_{1}\) used below.

\begin{lemma}\label{Lemma Homotopically Equivalent}
    If \(f\colon{X\to Y}\) and \(f'\colon{X'\to Y}\) are two projective presentations of an object \(Y\) then the equivalence relations \(\R{f}\two X\) and \(\R{f'}\two X'\) are homotopically equivalent.
\end{lemma}
\begin{proof}
    Any lifting of \(f\) over \(f'\) extends to an internal functor, pseudo-inverse to any internal functor obtained out of a lifting of \(f'\) over \(f\).
\end{proof}

\begin{convention}\label{Convention Comonadic Homology}
    The ordinary homology against which we compare the \Tsjech\ derived functors is the Barr--Beck comonadic homology, whose definition requires a comonad whose values are regular-projective objects. Such a comonad is canonical in the \defn{varietal context}, in which \(\C\) is a semi-abelian variety, \(\X\) a subvariety\footnote{That is, a subclass defined, in the sense of universal algebra, by additional identities; by Birkhoff's theorem~\cite{ACourseInUA} these are precisely the Birkhoff subcategories of \(\C\) (Definition~\ref{Definition Birkhoff Subcategory}), which is the source of that term.} of~\(\C\) and \(F\colon{\C\to \X}\) the reflector: then \(\C\) is monadic over \(\Set\), and the free/forgetful adjunction induces a comonad on \(\C\) by which every object is covered by a projective~\cite[Example~5.3]{EGVdL}. \emph{Whenever a homology object \(H_{n}(Y,F)\) occurs below, we place ourselves in this context.} Following~\cite[Definition~5.1]{EGVdL}---which extends the cotriple homology of Barr and Beck~\cite{Barr-Beck} to the semi-abelian setting---we then write
    \[
        H_{n}(Y,F)\coloneq H_{n-1}(F(P_{\bullet}))
    \]
    for the \(n\)th \defn{homology object} of \(Y\) with coefficients in \(F\), where \(n\geq 1\), the simplicial object \(P_{\bullet}\to Y\) is the canonical resolution of \(Y\) provided by this comonad, and \(H_{n-1}\) denotes the homology of the associated Moore complex. With this indexing one has \(H_{1}(Y,F)=F(Y)\), while \(H_{2}(Y,F)\) is the fundamental group \(\pi_{1}^{F}(Y)\); more generally \(H_{n+1}(Y,F)\) is computed by the Hopf formula of~\cite{EGVdL} in terms of \(n\)-fold central extensions. For groups with abelianisation as coefficients, \(H_{n}(Y,F)\) is the \(n\)th integral homology group of \(Y\)~\cite[Example~5.2]{EGVdL}. The same notation is used at every level of the tower of Galois structures of Section~\ref{section higher galois}: for an object \(A\) of~\(\Ext^{k}(\C)\), we write \(H_{n}(A,F_{k})\) for the homology objects formed in the same way from the level-\(k\) reflector \(F_{k}\) and the comonad induced on \(\Ext^{k}(\C)\)~\cite[Proposition~5.6]{EGVdL}; in particular \(H_{2}(A,F_{k})=\pi_{1}^{F_{k}}(A)\).

    The two numberings differ on purpose: the \Tsjech\ derived functors follow the simplicial degree, while \(H_{n}(Y,F)\) follows the indexing classical in the semi-abelian and group homology literature, in which \(H_{2}(Y,F)\) plays the role of the Schur multiplier. Hence the comparison below takes the shifted form \(\HH_{n}F(Y)\cong H_{n+1}(Y,F)\).
\end{convention}

Combining the previous results, we obtain

\begin{theorem}\label{Theorem Comparison Level 1}
    For any object \(Y\) in a semi-abelian category with enough projectives~\(\C\), its fundamental group \(\pi_{1}^{F}(Y)\) with respect to a Birkhoff reflector \(F\colon \C\to \X\) is isomorphic to \(\HH_{1}F(Y)\).
\end{theorem}
\begin{proof}
    Let \(f\colon{X\to Y}\) be a projective presentation of \(Y\). Then by Proposition~\ref{Proposition HH_1 as Homology of a Groupoid}, \(\HH_{1}F(Y)\) is the first homology of the groupoid \(F(\R{f})\two F(X)\). Now by Theorem~\ref{Theorem Interpretation Fundamental Group}, this groupoid is the Galois groupoid of the centralisation of \(f\). Hence the first homology in question is nothing but the fundamental group of \(Y\) with respect to the reflector \(F\).
\end{proof}

\begin{corollary}\label{Corollary Fundamental Group Projective}
    The fundamental group of a projective object is trivial.
\end{corollary}
\begin{proof}
    For a projective object \(P\) we have \(\HH_{1}F(P)=0\) by Proposition~\ref{Cech on Split Epi}, so that \(\pi_{1}^{F}(P)\cong\HH_{1}F(P)=0\) by Theorem~\ref{Theorem Comparison Level 1}.
\end{proof}

\begin{remark}\label{Remark Culot Everaert}
    In the varietal context (Convention~\ref{Convention Comonadic Homology}), the fundamental group \(\pi_{1}^{F}(Y)\) is the second homology object \(H_{2}(Y,F)\)~\cite{EGVdL}, so that Theorem~\ref{Theorem Comparison Level 1} yields \(H_{2}(Y,F)\cong \HH_{1}F(Y)\). This is the main result of~\cite{Culot-Everaert}, where Culot and Everaert obtain it directly, by simplicial means, and in the generality of our Theorem~\ref{Theorem Comparison General}: that of a semi-abelian category equipped with a comonad satisfying Quillen's conditions. The two papers were written independently: we learnt of their result only when the present article was essentially complete. Their proof computes the two sides by one and the same explicit formula: writing \(\del_{0}\), \(\del_{1}\), \(\del_{2}\colon{X^{3}_{Y}\to X^{2}_{Y}}\) for the faces of \(\Cech(f)\) in degree two, it is the quotient of \(\Ker(F(\pi_{0}))\cap \Ker(F(\pi_{1}))\) by the image under \(F(\del_{0})\) of \(\Ker(F(\del_{1}))\cap \Ker(F(\del_{2}))\). Their Moore complex is normalised with the differential induced by \(\del_{0}\) rather than by the last face; here the difference is immaterial, since by Proposition~\ref{Birkhoff Reflector Nerve} the Moore complex of \(F(\Cech(f))\) is trivial in degree two in either normalisation---the reverse of a nerve being the nerve of the opposite groupoid---while the numerator is the first homology of the Galois groupoid.

    The higher degrees are explicitly left open in~\cite{Culot-Everaert}, which singles out the induction along the tower of Galois structures of~\cite{EGVdL} as the route towards them: it is precisely that induction which Section~\ref{section induction} carries out, and the Galois-theoretic argument given here is the one it iterates. In the other direction, Culot and Everaert compare the first \Tsjech\ derived functor with the first non-additive left derived functor of~\cite{CRVdL}, computed from a chain resolution rather than from a simplicial one.
\end{remark}

Everything done up to now carries over to the relative setting of higher extensions developed in~\cite{EGVdL}, in which the regular epimorphisms are replaced by an abstract class \(\E\) of extensions---defined inductively, starting from the regular epimorphisms of \(\C\)---and \(F\) by a strongly \(\E\)-Birkhoff reflector. Since the induction of Section~\ref{section induction} invokes the results of Sections~\ref{section basic concepts}--\ref{section reflection nerve} at every level of the tower of Galois structures recalled in the next section, let us make explicit what this transfer requires. Beyond the axioms of a Galois structure, which are formulated relatively to \(\E\) from the start, our proofs use: that the surrounding category is homological, and that kernel pairs of extensions, and pullbacks of extensions along extensions, are again extensions; the strong \(\E\)-Birkhoff property of \(F\), through the double extensions it provides and through Lemma~\ref{Lemma Kernels Double Extension}, whose proof appeals to nothing but the pullback-stability of \(\E\); enough \(\E\)-projectives, wherever projective presentations occur; and, for the fundamental group, the weak universality of the centralisation of an \(\E\)-projective presentation---Theorem~5.1 of~\cite{Janelidze:Hopf}, which is established there for an abstract pointed Galois structure, projectivity being taken relatively to the class of extensions. Exactness is not needed: only the \(H_{0}\)-part of Lemma~\ref{Lemma Exact Functor Groupoid Homology} uses it, and that lemma is applied only between semi-abelian categories (Theorem~\ref{Loday Theorem}). It is for the sake of this transfer that Proposition~\ref{Proposition Characterisation Nerve} was stated for homological categories.

At each level \(k\) of the tower, the requirements just listed are supplied by~\cite{EGVdL}: the categories \(\Ext^{k}(\C)\) are homological, with pullbacks along \(\E^{k}\)-extensions computed degree-wise and \(\E^{k}\) stable under pullbacks and closed under composition~\cite[Proposition~3.5]{EGVdL}; \(\CExt^{k}(\C)\) is a strongly \(\E^{k}\)-Birkhoff subcategory of \(\Ext^{k}(\C)\)~\cite[Lemma~4.3]{EGVdL}, which makes the Galois structure \(\Gamma_{k}\) admissible~\cite[Proposition~2.6]{EGVdL}; the level-\(k\) form of Lemma~\ref{Lemma Pullback Preservation Split} is Lemma~4.4 of~\cite{EGVdL}, stated there for every level at once; and \(\Ext^{k}(\C)\) has enough \(\E^{k}\)-projectives whenever \(\C\) carries a comonad satisfying Quillen's conditions---in particular in the varietal context of Convention~\ref{Convention Comonadic Homology}---by~\cite[Proposition~5.6]{EGVdL}, extension-projectivity admitting the vertex-wise description of~\cite[Lemma~5.5]{EGVdL}. Results of these sections invoked below at level \(k\) are always meant in this relative reading.

\section{Higher Galois theory, higher groupoids}\label{section higher galois}

To proceed with higher Galois theory, we define inductively \(\Arr^n(\C)\) by setting \(\Arr^0(\C)=\C\) and \(\Arr^{n+1}(\C)\) the category whose objects are morphisms in \(\Arr^n(\C)\) and whose morphisms are commutative squares in \(\Arr^n(\C)\). An \defn{\(n\)-extension} is an extension in \(\Arr^{n-1}(\C)\); these determine a full subcategory \(\Ext^n(\C)\) of \(\Arr^n(\C)\). Following~\cite[Proposition~3.5]{EGVdL}, we write \(\E^n\) for the class of \((n+1)\)-extensions, so that \(\E^0\) is the class of regular epimorphisms and \(\E^1\) that of double extensions. At the end of Section~\ref{section basic concepts} we constructed an admissible Galois structure \(\Gamma_1\coloneq(\Ext(\C),\CExt(\C),\E^1,\Z^1,F_1,{\subseteq_1})\) in which \(\CExt(\C)\) is a strongly \(\E^1\)-Birkhoff subcategory of \(\Ext(\C)\). \(\Ext(\C)\) is no longer semi-abelian, but it remains homological~\cite{EGVdL}; although it is not closed under limits in \(\Arr(\C)\) (Example~\ref{Example Limits Ext}), the pullbacks along double extensions and the kernel pairs that the construction uses are computed degree-wise (Lemma~\ref{Lemma Kernel Pairs Ext} and, for all levels of the tower at once, \cite[Proposition~3.5]{EGVdL}), while the cokernels it needs are supplied by Proposition~\ref{Proposition Colimits Ext} and, again at every level, by~\cite[Section~3]{EGVdL}; so the construction may be iterated. We obtain a sequence of admissible Galois structures \(\Gamma_n\coloneq(\Ext^n(\C),\CExt^n(\C),\E^n,\Z^n,F_n,{\subseteq_n})\) in which \(\CExt^n(\C)\) is a strongly \(\E^n\)-Birkhoff subcategory of \(\Ext^n(\C)\). All of this is explained in more detail in~\cite{EGVdL}.

\begin{notation}\label{Notation iota}
    For a pointed category \(\D\), we write \(\iota\colon{\D\to\Arr(\D)}\) for the functor sending an object \(A\) to the morphism \((A\to 0)\) with zero codomain, and a morphism to the evident commutative square; its \(n\)-fold iterate is written \(\iota^{n}\colon{\D\to\Arr^{n}(\D)}\). A first homology object has trivial codomain, so is canonically of this form; we use \(\iota\) to record at which level of the tower \(\Arr^{n}(\C)\) we consider such an object.
\end{notation}

\begin{definition}
    Given a non-negative integer \(n\), we define inductively an \defn{\((n+1)\)-fold groupoid} as a groupoid internal to the category of \(n\)-fold groupoids, with \(\C\) simply being the category of \(0\)-fold groupoids.
\end{definition}

\begin{notation}\label{Notation von Neumann}
    We treat a natural number \(n\) as the set of its predecessors, as in the von Neumann construction of the natural numbers: \(n=\{0,\dots,n-1\}\), so that \(i\in n\) means \(0\leq i\leq n-1\). This convention is used in the description of \(n\)-fold groupoids below and in Theorem~\ref{Loday Theorem}.
\end{notation}

As recalled in Section~\ref{section reflection nerve}, in a Mal'tsev category every internal category is automatically a groupoid; moreover, a reflexive graph admits at most one multiplication~\cite{Carboni-Pedicchio-Pirovano}. Being a groupoid is thus a \emph{property} of the underlying reflexive graph, rather than additional structure on it. Since categories of \(n\)-fold groupoids in \(\C\) are themselves finitely complete Mal'tsev categories---by induction, they are closed under finite limits in a category of diagrams in \(\C\)---this observation applies in each of the \(n\) directions of an \(n\)-fold groupoid. Hence no information is lost if we summarise a groupoid \(G\) as the data of two objects \(G_0\) and \(G_1\), together with two regular epimorphisms \(d_0\), \(d_1\colon G_1\to G_0\) with a common section \(s_0\colon G_0\to G_1\): an \(n\)-fold groupoid may then be seen as a cube of dimension \(n\) in \(\C\), with edges being two regular epimorphisms with a common section, satisfying certain conditions. Every vertex can be identified as an object \(G_t\) where \(t\) is an \(n\)-tuple containing \(0\) and \(1\). This notation is fixed as follows: if \(G\coloneq({X\two Y})\) is an \((n+1)\)-fold groupoid, then every object \(G_t\) in the \(n\)-fold groupoid \(X\) has \(1\) as its last coordinate, while every object \(G_t\) in the \(n\)-fold groupoid \(Y\) has \(0\) as its last coordinate. For example, if \(G\coloneq({X\two Y})\) is a double groupoid, then we have the following diagram:
\[
    \xymatrix{G_{11} \ar@<.5ex>[r] \ar@<-.5ex>[r] \ar@<.5ex>[d] \ar@<-.5ex>[d] & G_{10} \ar@<.5ex>[d] \ar@<-.5ex>[d] \\
        G_{01} \ar@<.5ex>[r] \ar@<-.5ex>[r] & G_{00}}
\]
Here \(X\) is on the left-hand side, and \(Y\) on the right-hand side. For an \(n\)-fold groupoid \(G\), the two structure maps going from \(G_{1\cdots 1}\) to \(G_t\), where \(t\) is the \(n\)-tuple obtained from \(1\cdots 1\) by replacing the \(i\)-th coordinate with \(0\), are denoted by \(d^i_0\) and \(d^i_1\) for \(i\in n\).
\begin{definition}\label{Definition Higher Nerve}
    The \defn{nerve} of an \(n\)-fold groupoid \(G\) is the \(n\)-fold simplicial object defined inductively as follows:
    \begin{itemize}
        \item if \(n=1\), then \(\Nerve^1(G)\) is as in the beginning of Section~\ref{section reflection nerve};
        \item if \(n\geq 1\), and \(G\coloneq({X\two Y})\) is an \((n+1)\)-fold groupoid, then taking nerves of the \(n\)-fold groupoids \(X\) and \(Y\) gives a groupoid of \(n\)-fold simplicial objects, and we define \(\Nerve^{n+1}(G)\) as the nerve of this groupoid.
    \end{itemize}
\end{definition}

\begin{remark}\label{Remark Nerve Symmetry}
    Definition~\ref{Definition Higher Nerve} treats the directions of an \((n+1)\)-fold groupoid in a chosen order, the direction of the outer groupoid \({X\two Y}\) being dealt with last; up to canonical isomorphism, however, the result is independent of this choice. It suffices that two consecutive directions may be interchanged, which is the case of a double groupoid \(G\), possibly internal to a category of multiple groupoids. Here the two possible orders yield, in each double degree \((p,q)\), one and the same object of \(p\times q\) grids of composable squares in \(G\): the nerve functor is constructed out of finite limits, hence preserves them, while limits of groupoids and of simplicial objects are computed levelwise; so a comparison of universal properties makes the two double simplicial objects canonically isomorphic, with corresponding face and degeneracy operators. We use this symmetry implicitly in Section~\ref{section induction}, where homology is taken in the various directions of a multiple groupoid, one by one in turn. A symmetry of the same kind for higher coverings is established in~\cite{RenaudVdL}.
\end{remark}

Lemma~\ref{Lemma Nerve exact} and
Proposition~\ref{Birkhoff Reflector Nerve} now extend as follows.

\begin{lemma}\label{Lemma Higher Nerve Exact}
    For each \(n\geq 1\), the \(n\)-fold nerve functor \(\Nerve^{n}\), from \(n\)-fold groupoids in \(\C\) to \(n\)-fold simplicial objects in \(\C\), preserves limits and regular epimorphisms; in particular, it is exact.
\end{lemma}
\begin{proof}
    Induction on \(n\), the case \(n=1\) being Lemma~\ref{Lemma Nerve exact}. Since \(\s^{n}\C\) is again semi-abelian, with limits and regular epimorphisms computed levelwise, applying \(\Nerve^{n}\) levelwise takes a groupoid of \(n\)-fold groupoids to a groupoid in \(\s^{n}\C\), preserving limits and regular epimorphisms by the induction hypothesis. By Definition~\ref{Definition Higher Nerve}, \(\Nerve^{n+1}\) is this functor followed by Lemma~\ref{Lemma Nerve exact} for \(\s^{n}\C\).
\end{proof}

\begin{proposition}\label{Proposition Birkhoff Reflector Higher Nerve}
    For any \(n\)-fold groupoid \(C\) in \(\C\), the degree-wise image \(F(C)\) is an \(n\)-fold groupoid in \(\X\), and \(F(\Nerve^{n}(C))\cong\Nerve^{n}(F(C))\).
\end{proposition}
\begin{proof}
    In each multidegree, \(\Nerve^{n}(C)\) is an iterated pullback of split epimorphisms along split epimorphisms, all of which \(F\) preserves by Lemma~\ref{Lemma Pullback Preservation Split}; hence \(F(\Nerve^{n}(C))\cong\Nerve^{n}(F(C))\). As in Proposition~\ref{Birkhoff Reflector Nerve}, the reflexive graph \(F(C)\) is then multiplicative in each of its \(n\) directions, hence an \(n\)-fold groupoid since \(\X\) is Mal'tsev~\cite{Carboni-Pedicchio-Pirovano}.
\end{proof}

\section{Loday's theorem}\label{section loday}

As proved by the first author in~\cite{Afsa-Spectral}, mutatis mutandis, Quillen's argument in~\cite{DGQlln1966} leading to two spectral sequences for a double simplicial group is fully valid in arbitrary semi-abelian categories. We use those to prove a categorical version of \emph{Loday's theorem}---essentially Proposition~3.4 in~\cite{Loday}, of which a purely algebraic version is Proposition~14 of~\cite{Donadze-Inassaridze-Porter}---needed in our analysis of \Tsjech\ homology.

\begin{theorem}[Quillen spectral sequences]\label{Quillen Theorem}
    If \(G\) is a double simplicial object in a semi-abelian category, then there are two spectral sequences
    \begin{align*}
        E^{2}_{pq}=H^{h}_{p}H^{v}_{q}G & \To H_{p+q}(\Diag G), \\
        E^{2}_{pq}=H^{v}_{p}H^{h}_{q}G & \To H_{p+q}(\Diag G).
    \end{align*}
    Here \(\Diag G\) denotes the diagonal of \(G\), and \(H^h_pG\) and \(H^v_qG\) its horizontal and vertical homology, respectively.\noproof
\end{theorem}

Only the first of the two is used in what follows.

The proof follows Quillen's: the spectral sequence is derived from the exact couple which the Moore complex builds out of the long exact homology sequences of \(G\), read as a simplicial object in simplicial objects; see~\cite{Afsa-Spectral} for the details. The single point at which the semi-abelian setting genuinely differs from the abelian one is that the passage from an exact couple to a spectral sequence is no longer automatic---it requires the differentials to be \emph{normal} morphisms. This normality is secured by the fact that the homology objects of a simplicial object are abelian in positive degrees~\cite[Theorem~5.5]{EverVdL2}: every entry of the second page away from the origin is then an abelian object, and a morphism between abelian objects is normal, since its image is a subobject of an abelian object, and any such subobject is normal~\cite[Corollary~3.2.17]{Borceux-Bourn}; cf.~\cite[Proposition~2.11]{Afsa-Spectral}. Convergence is automatic, the spectral sequence being a first-quadrant one. Recall that convergence means that there is a filtration
\[
    0=F_{-1}H_n\leq F_0H_n\leq F_1H_n\leq \cdots \leq F_nH_n= H_n(\Diag G)
\]
by (normal) subobjects of \(H_n(\Diag G)\)---which, in the degrees \(n\geq 1\) relevant to us, is an abelian object~\cite{EverVdL2}, so that all of its subobjects are normal---together with natural isomorphisms \(E^\infty_{pq}\cong F_pH_{p+q}/F_{p-1}H_{p+q}\), where \(n=p+q\). Each \(E^{\infty}_{pq}\) is in turn a subquotient of the corresponding \(E^{2}\)-entry, so the graded pieces of \(H_n(\Diag G)\) are subquotients of the second page.

We shall apply this in the following situation: the second page vanishes outside the block of bidegrees \((p,q)\) with \(0\leq p\leq n\) and \(q\in\{0,1\}\). On the diagonal \(p+q=n+1\) the only entry inside the block is then \((n,1)\), and no differential touches it. Indeed, with the usual homological bidegree \(d^{r}_{p,q}\colon{E^{r}_{p+r,\,q-r+1}\to E^{r}_{p,q}}\), a differential out of \((n,1)\) lands in \((n-r,r)\) with \(r\geq 2\), while one into it comes from \((n+r,\,2-r)\), which for \(r\geq 3\) lies outside the first quadrant and for \(r=2\) is the entry \((n+2,0)\)---all outside the block, and each \(E^{r}\)-entry is a subquotient of the corresponding \(E^{2}\)-entry. Hence \(E^{\infty}_{n,1}=E^{2}_{n,1}\) is the only nonzero graded piece of the filtration of \(H_{n+1}(\Diag G)\), so that
\[
    H_{n+1}(\Diag G)\cong E^{2}_{n,1}\text{,}
\]
while for \(j\geq n+2\) the whole \(j\)th diagonal avoids the block, whence \(H_{j}(\Diag G)=0\).

\begin{theorem}[Loday's theorem]\label{Loday Theorem}
    In a semi-abelian category, given any internal \((n+1)\)-fold groupoid \(G\), we have
    \[
        H_{n+1}\Diag(\Nerve^{n+1}(G))=\bigcap_{i\in n+1}(\Ker(d^{i}_{0})\cap \Ker(d^{i}_{1}))
    \]
    and \(H_{j}\Diag(\Nerve^{n+1}(G))=0\) for all \(j\geq n+2\).
\end{theorem}
\begin{proof}
    We prove the result by induction on \(n\geq 0\). The case \(n=0\) being clear, we take \(n\geq 1\), assume that
    \[
        H_{n}\Diag(\Nerve^{n}(X))=\bigcap_{i\in n}(\Ker(d^{i}_{0})\cap \Ker(d^{i}_{1}))
    \]
    and \(H_{k}\Diag(\Nerve^{n}(X))=0\) for all \(k\geq n+1\), for any \(n\)-fold groupoid \(X\) in any semi-abelian category, and prove the result for a groupoid
    \[
        G=(d_{0},d_{1}\colon{X\two Y})
    \]
    of \(n\)-fold groupoids \(X\), \(Y\) in a semi-abelian category \(\C\). We consider its nerve as a ``vertical'' simplicial object of \(n\)-fold groupoids also written \(G\), so that \({G_{1}=X}\) and \(G_{0}=Y\). On the other hand, applying \(\Diag\circ \Nerve^{n}\) to the objects of the simplicial object~\(G\), we find a ``horizontal'' simplicial object \(\Diag(\Nerve^{n}(G_{j}))\) in~\(\C\), one in each degree \(j\) of the simplicial \(n\)-fold groupoid \(G\). We thus obtain the double simplicial object \(S\) in \(\C\) with \(S_{ij}=\Diag(\Nerve^{n}(G_{j}))_{i}\): its \(j\)th row \((S_{ij})_{i}\) is the simplicial object \(\Diag(\Nerve^{n}(G_{j}))\), while its vertical structure is obtained by applying \(\Diag\circ\Nerve^{n}\) to the faces and degeneracies of the simplicial object \(G\) and taking degree-\(i\) components. Note that its diagonal is precisely the simplicial object \(\Diag (\Nerve^{n+1}(G))\) in whose \((n+1)\)st homology we are interested.

    Further note that for each index \(i\) we have a reflexive graph \(S_{i1}\two S_{i0}\). First of all, this reflexive graph is a groupoid, since the functor \(\Diag\circ \Nerve^{n}\) preserves all finite limits. In fact, it is an exact functor between semi-abelian categories---the category of \(n\)-fold groupoids in a semi-abelian category being again semi-abelian, for \(n=1\) by~\cite{Bourn-Gran} and in general by induction~\cite[Section~1]{EG-honfg}---as a composite of two functors which are exact (Lemma~\ref{Lemma Higher Nerve Exact} for \(\Nerve^{n}\), while the diagonal \(\Diag\) is exact because limits and colimits in a functor category are pointwise). For the same reason, the ``vertical'' simplicial object \(S^{v}_{i}=(S_{ij})_{j}\) is its nerve. As a consequence of Proposition~\ref{Proposition Characterisation Nerve}, the Moore complex of \((S_{ij})_{j}\) vanishes in degrees greater than or equal to 2.

    Taking homology vertically therefore gives a first page concentrated in the rows \(q\in\{0,1\}\), the entries \(H_{0}S^{v}_{i}\) and \(H_{1}S^{v}_{i}\) being the homologies of the groupoid \({S_{i1}\two S_{i0}}\). Taking homology horizontally, we obtain the second page,
    \[
        \xymatrix@=1em{
        \ddots & \vdots & \vdots & & \vdots &\vdots&\vdots\\
        \cdots & 0 & 0 & \cdots & 0 & 0 & 0\\
        \cdots & 0 & E^{2}_{n1} & \cdots & E^{2}_{21} & E^{2}_{11} & E^{2}_{01}\\
        \cdots & 0 & E^{2}_{n0} & \cdots & E^{2}_{20} & E^{2}_{10} & E^{2}_{00}.}
    \]
    The entries outside the block indexed by the elements of \((n+1)\times 2\) are zero because, \(\Diag\circ \Nerve^{n}\) being an exact functor, by Lemma~\ref{Lemma Exact Functor Groupoid Homology} the entry \(E^{2}_{ij}=H_{i}H_{j}S\) may be obtained as \(H_{i}\circ\Diag\circ \Nerve^{n}\) of the \(j\)th homology of the groupoid \(G\)---note that this \(H_{j}G\) is an \(n\)-fold groupoid itself---which by the induction hypothesis vanishes in degrees \(i\geq n+1\). Similarly, \(E^{2}_{n1}\) is \(H_{n}\circ\Diag\circ \Nerve^{n}\) of the first homology \(\Ker(d_{0})\cap \Ker(d_{1})\) of the groupoid \(G=(d_{0},d_{1}\colon{X\two Y})\), which by the induction hypothesis is nothing but the intersection
    \[
        \bigl(\Ker(d^{n}_{0})\cap \Ker(d^{n}_{1})\bigr) \cap \bigcap_{i\in n}\bigl(\Ker(d^{i}_{0})\cap \Ker(d^{i}_{1})\bigr)= \bigcap_{i\in n+1}\bigl(\Ker(d^{i}_{0})\cap \Ker(d^{i}_{1})\bigr).
    \]
    The second page is thus concentrated in the two rows \(q\in\{0,1\}\) and the columns \(0\leq p\leq n\), so that the discussion preceding Theorem~\ref{Loday Theorem} gives \(H_{n+1}\Diag(\Nerve^{n+1}(G))\cong E^{2}_{n,1}\) (which is precisely the intersection computed above) and \(H_{j}\Diag(\Nerve^{n+1}(G))=0\) for \(j\geq {n+2}\).
\end{proof}

\section{Induction in \Tsjech\ homology}\label{section induction}

The aim of this section is to extend the comparison of Theorem~\ref{Theorem Comparison Level 1} from the first degree to all degrees, by an induction along the tower of Galois structures of~\cite{EGVdL}. The underlying mechanism relates the higher \Tsjech\ derived functors of \(F\) to those of its centralisation \(F_{1}\): its first instance, Proposition~\ref{Proposition Induction Level 1}, gives
\[
    (\HH_{2}F(Y)\to 0)\cong\HH_{1}F_{1}(f)
\]
for a projective presentation \(f\) of \(Y\), while the general step takes the form of Proposition~\ref{Proposition Shift}. Since the fundamental groups satisfy the same kind of shift, this yields the main result, Theorem~\ref{Theorem Comparison Level n}, in all degrees.

Besides the degree shift fixed in Convention~\ref{Convention Comonadic Homology}, two pieces of bookkeeping are at work in what follows. First, the argument runs along the tower \(\C\), \(\Ext(\C)\), \(\Ext^{2}(\C)\), \dots\ of~\cite{EGVdL}, descending one level at each step as it passes from \(F_{k}\) to \(F_{k-1}\), while adding one further groupoid direction \(H_{1}\). Second, a first homology object has trivial codomain, and the embedding \(\iota^{k}\) of Notation~\ref{Notation iota} merely records the level \(k\) at which such an object is read; it may be taken as the annotation ``at level~\(k\)'' and ignored on a first reading.

\subsection{Higher \Tsjech\ derived functors}

Suppose that \(\C\) is a semi-abelian category with enough projectives and that \(F\colon{\C\to \X}\) is a functor to a homological category~\(\X\). The higher \Tsjech\ derived functors \(\HH_{n+1}F(Y)\) are obtained inductively as follows. We already know that \(\HH_{1}F(Y)\) is the first homology of \(F(\Cech (f))\), which is the first homology of the groupoid \(F(\R{f})\two F(X)\), where \(f\colon {X\to Y}\) is a projective presentation of \(Y\).

\begin{definition}\label{Definition Higher Presentation}
    A \defn{\(0\)-fold projective presentation} of \(Y\) is the object \(Y\) itself. For \(n\geq 0\), an \defn{\((n+1)\)-fold projective presentation} of \(Y\) is an \((n+1)\)-extension~\(P\), viewed as an extension \(P\colon{\Dom P\to\Cod P}\) in \(\Ext^{n}(\C)\), whose codomain \(\Cod P\) is an \(n\)-fold projective presentation of \(Y\) and whose domain \(\Dom P\) is extension-projective in the sense of~\cite[Definition~5.4]{EGVdL}---equivalently~\cite[Lemma~5.5]{EGVdL}, every object of \(\C\) occurring as a vertex of the \(n\)-cube \(\Dom P\) is projective. In particular, a \(1\)-fold projective presentation is a projective presentation \(f\colon{X\to Y}\) in the sense used so far.
\end{definition}

In the varietal context of Convention~\ref{Convention Comonadic Homology} (and, more generally, whenever each category \(\Ext^{k}(\C)\) has enough \(\E^{k}\)-projectives) every object \(Y\) admits an \((n+1)\)-fold projective presentation for every \(n\); a canonical one is provided by the comonadic resolution~\cite[Corollary~5.7]{EGVdL}. The same construction, one level at a time, extends any \(k\)-fold projective presentation of \(Y\) to an \((n+1)\)-fold one; in particular, any projective presentation of \(Y\) is an edge of a \(2\)-fold projective presentation of \(Y\), as in Diagram~\eqref{Blokske} below.

We adopt the following conventions for the multi-directional constructions, using the direction labels of Section~\ref{section higher galois}. Given an \((n+1)\)-fold projective presentation \(P\) of \(Y\), applying the kernel-pair functor \(\Eq\) in each of the \(n+1\) directions \(i\in n+1\) produces the \((n+1)\)-fold groupoid \(\Eq^{n+1}(P)\), whose projections in the \(i\)th direction are written \(d^{i}_{0}\) and \(d^{i}_{1}\); applying the \Tsjech\ functor \(\Cech\) in each direction produces the \((n+1)\)-fold simplicial object \(\Cech^{n+1}(P)\), which by Lemma~\ref{Cech Is Nerve} is the \((n+1)\)-fold nerve \(\Nerve^{n+1}(\Eq^{n+1}(P))\). Finally, the \defn{iterated first homology} \(H_{1}^{n+1}\) is obtained by taking the first homology \(H_{1}\) in each of the \(n+1\) groupoid directions in turn; by Remark~\ref{Remark Nerve Symmetry}, up to canonical isomorphism, the result does not depend on the order in which the directions are treated.

By definition, the object \(\HH_{n+1}F(Y)\) is the \((n+1)\)st homology of the diagonal of \(F(\Cech^{n+1}(P))\), for a chosen \((n+1)\)-fold projective presentation \(P\) of \(Y\); that the choice is immaterial is the content of Corollary~\ref{Corollary Higher Cech Well Defined} below. Combining Lemma~\ref{Cech Is Nerve} and Proposition~\ref{Proposition Birkhoff Reflector Higher Nerve} with Theorem~\ref{Loday Theorem} tells us that \(\HH_{n+1}F(Y)\) can be obtained as
\[
    \bigcap_{i\in n+1}(\Ker(F(d^{i}_{0}))\cap \Ker(F(d^{i}_{1})))
\]
where the \(d_{0}^{i}\), \(d_{1}^{i}\) are the projections of the \((n+1)\)-fold groupoid \(F(\Eq^{n+1}(P))\). This may be seen as the result of successively taking first homologies, each time on a groupoid pointing in a different direction, so that we obtain the following higher-order version of Proposition~\ref{Proposition HH_1 as Homology of a Groupoid}:

\begin{proposition}\label{Proposition HH_n+1 as Homology of a Groupoid}
    For any \((n+1)\)-fold projective presentation \(P\) of an object~\(Y\) we have that
    \[
        \HH_{n+1}F(Y)\cong H_{1}^{n+1}F(\Eq^{n+1}(P))\text{,}
    \]
    the left-hand side being formed from \(P\).
    \noproof
\end{proposition}

It remains to be seen that the right-hand side (and with it the \Tsjech\ derived functor \(\HH_{n+1}F\) itself) does not depend on the chosen \((n+1)\)-fold projective presentation \(P\) of \(Y\). The degree-one argument of Lemma~\ref{Lemma Homotopically Equivalent} does not lift: a morphism \(\phi\colon{P\to P'}\) of \((n+1)\)-fold projective presentations over \(Y\) is the identity at the corner \(Y\) alone, so that a composite \(\psi\phi\) is an endomorphism of \(P\) which is nowhere else an identity, and \(\Eq^{n+1}(\psi\phi)\) admits no internal natural transformation to the identity in any one of the \(n+1\) groupoid directions. We therefore obtain the independence in the opposite order, as a consequence of the comparison with comonadic homology rather than as a preliminary to it (Corollary~\ref{Corollary Higher Cech Well Defined} below): the induction of the present section starts from an \emph{arbitrary} \((n+1)\)-fold projective presentation, while the comonadic homology object it arrives at involves no choice at all. The same order of deduction is followed in~\cite{EGVdL} for the higher Hopf formulae.

\begin{remark}\label{Remark Classical Hopf Formula}
    In the case of abelianisation, the higher \Tsjech\ derived functors \(\HH_{n}\ab\) recover the higher Hopf formulae of~\cite{Donadze-Inassaridze-Porter}; the role played there by a direct computation from an \(n\)-fold presentation is here taken over by our categorical version of Loday's theorem (Theorem~\ref{Loday Theorem}).
\end{remark}

\subsection{Case study: \texorpdfstring{\(n=2\)}{n=2}}

It follows from Theorem~\ref{Loday Theorem} that, if a \(2\)-fold equivalence relation with projections \(d^{i}_{0}\), \(d^{i}_{1}\) is induced by a \(2\)-fold presentation of an object \(Z\) as in
\begin{equation}\label{Blokske}
    \vcenter{\xymatrix{\R{d}\square \R{c} \ar@<.5ex>[r]^-{d^{1}_1} \ar@<-.5ex>[r]_-{d^{1}_0} \ar@<.5ex>[d]^-{d^{0}_1} \ar@<-.5ex>[d]_-{d^{0}_0} & \R{d} \ar[r]^-{p} \ar@<.5ex>[d]^-{d_1} \ar@<-.5ex>[d]_-{d_0} & \R{g} \ar@<.5ex>[d]^-{g_{1}} \ar@<-.5ex>[d]_-{g_{0}} \\
    \R{c} \ar@<.5ex>[r]^-{c_1} \ar@<-.5ex>[r]_-{c_0} \ar[d]_-{r} & X \ar[r]^-{c} \ar[d]_-{d} & C \ar[d]^-{g}\\
    \R{f} \ar@<.5ex>[r]^-{f_{1}} \ar@<-.5ex>[r]_-{f_{0}} & D \ar[r]_-{f} & Z,}}
\end{equation}
then
\[
    \HH_{2}F(Z)=\bigcap_{i\in 2}(\Ker(F(d^{i}_{0}))\cap \Ker(F(d^{i}_{1}))).
\]
The bottom-right square \((c,f)\colon{d\to g}\) in~\eqref{Blokske} is a \(2\)-fold projective presentation of~\(Z\); here \(g\colon{C\to Z}\) is the projective presentation of \(Z\) we resolve. The objects \(\R{c}\), \(\R{d}\), \(\R{f}\) and \(\R{g}\) are the kernel pairs of \(c\), \(d\), \(f\) and \(g\), with projections \(c_{i}\), \(d_{i}\), \(f_{i}\), \(g_{i}\) for \(i\in 2\); \(\R{d}\square\R{c}\) is their common kernel pair in the two directions, with projections \(d^{0}_{i}\) towards \(\R{c}\) and \(d^{1}_{i}\) towards \(\R{d}\); and \(p\), \(r\) are the morphisms induced between the kernel pairs. We compare this intersection with \(\HH_{1}F_{1}(g)\).

Applying \(F\) to the top-left double equivalence relation of~\eqref{Blokske} yields an internal double groupoid---Diagram~\eqref{n-Double} below with \(n=1\)---of which, by Theorem~\ref{Loday Theorem}, the object \(\HH_{2}F(Z)\) is \(H_{2}\) of the diagonal of the double nerve. Applying instead \(F_{1}\) to the bottom part of~\eqref{Blokske}, as in Diagram~\eqref{n-H1} below with \(n=1\), and taking kernels to the left, we obtain \(\HH_{1}F_{1}(g)\), whose codomain is trivial since \((\R{f},f_{0},f_{1})\) is an equivalence relation. Here the square \((c,f)\colon{d\to g}\) is a projective presentation of the object \(g\) of \(\Ext(\C)\), a truncation of the \(2\)-fold projective presentation~\eqref{Blokske}, so that \(\HH_{1}F_{1}(g)\) is, in the Galois structure \(\Gamma_{1}\), the first homology of \(F_{1}\) applied to its kernel pair, which is computed degree-wise by Lemma~\ref{Lemma Kernel Pairs Ext}.

That these two intersections coincide is the content of the Key Lemma below (Lemma~\ref{Key Lemma}) in the case \(n=1\): the same chain of isomorphisms, of which the present situation is the prototype, identifies the domain \(\Ker(\overline{c_0})\cap\Ker(\overline{c_1})\) of \(\HH_{1}F_{1}(g)\) with \(\HH_{2}F(Z)\). Together with the identifications above, this gives \((\HH_{2}F(Z)\to 0)\cong\HH_{1}F_{1}(g)\). We thus obtain:

\begin{proposition}\label{Proposition Induction Level 1}
    For any projective presentation \(f\colon{X\to Y}\) of an object \(Y\), there is an isomorphism \((\HH_{2}F(Y)\to 0)\cong\HH_{1}F_{1}(f)\), the left-hand side being formed from any \(2\)-fold projective presentation of \(Y\) with codomain edge \(f\).
    \noproof
\end{proposition}

For the comparison with comonadic homology we also need the shift which the homology objects themselves undergo along the tower; we recall it from~\cite{EGVdL}, in our notation.

\begin{theorem}[{\cite[Theorem~6.4]{EGVdL}}]\label{Theorem EGVdL Shift}
    In the context of Convention~\ref{Convention Comonadic Homology}, for any \(n\geq 1\), \(k\geq 2\) and any projective presentation \(f\) of an \((n-1)\)-extension \(A\), there is an isomorphism \(H_{k}(f,F_{n})\cong \iota H_{k+1}(A,F_{n-1})\).\noproof
\end{theorem}

\begin{theorem}\label{Theorem Comparison Level 2}
    In the varietal context of Convention~\ref{Convention Comonadic Homology} (\(\C\) a semi-abelian variety and \(F\colon{\C\to \X}\) the reflector onto a subvariety) the third homology \(H_{3}(Y,F)\) of an object \(Y\) is isomorphic to \(\HH_{2}F(Y)\).
\end{theorem}
\begin{proof}
    Let \(f\colon{X\to Y}\) be a projective presentation of \(Y\). Combining Proposition~\ref{Proposition Induction Level 1} with Theorem~\ref{Theorem Comparison Level 1}, applied to the reflector \(F_{1}\), and with Theorem~\ref{Theorem EGVdL Shift}, we have that
    \[
        (\HH_{2}F(Y)\to 0)\cong\HH_{1}F_{1}(f)\cong H_{2}(f,F_{1})\cong (H_{3}(Y,F)\to 0),
    \]
    from which the claim follows. The middle isomorphism is Theorem~\ref{Theorem Comparison Level 1} for the reflector \(F_{1}\), which identifies \(\HH_{1}F_{1}(f)\) with the fundamental group
    \[
        \pi_{1}^{F_{1}}(f)=H_{2}(f,F_{1})\text{;}
    \]
    the last is Theorem~\ref{Theorem EGVdL Shift} with \(n=1\) and \(k=2\), the presentation \(f\) being a presentation of the \(0\)-extension \(Y\).
\end{proof}

From the above argument we may extract a technical result which will allow us to obtain a proof of the general case.

\subsection{A key lemma}

Let \(Y\) be an object of \(\C\). For \(n\geq 1\), we let the bottom-right square in Figure~\ref{n-Blokske} be an \((n+1)\)-fold projective presentation \(P\) of \(Y\).
\begin{figure}[tb]
    \(
    \xymatrix{\R{d}\square \R{c} \ar@<.5ex>[r]^-{d^{n}_1} \ar@<-.5ex>[r]_-{d^{n}_0} \ar@<.5ex>[d]^-{d^{n-1}_1} \ar@<-.5ex>[d]_-{d^{n-1}_0} & \R{d} \ar[r]^-{p} \ar@<.5ex>[d]^-{d_1} \ar@<-.5ex>[d]_-{d_0} & \R{g} \ar@<.5ex>[d]^-{g_{1}} \ar@<-.5ex>[d]_-{g_{0}} \\
    \R{c} \ar@<.5ex>[r]^-{c_1} \ar@<-.5ex>[r]_-{c_0} \ar[d]_-{r} & X \ar[r]^-{c} \ar[d]_-{d} & C \ar[d]^-{g}\\
    \R{f} \ar@<.5ex>[r]^-{f_{1}} \ar@<-.5ex>[r]_-{f_{0}} & D \ar[r]_-{f} & Z}
    \)
    \caption{An \((n+1)\)-fold projective presentation \(P\) of \(Y\) with its double kernel pair \(\R{d}\square \R{c}\)}\label{n-Blokske}
\end{figure}
Each entry of the diagram in Figure~\ref{n-Blokske} is an \((n-1)\)-fold extension, so that Figure~\ref{n-Blokske} displays a diagram in \(\Ext^{n-1}(\C)\) whose bottom-right square \(P\) is a \(2\)-fold extension there; the bottom-right corner \(Z\), the codomain of \(P\) in both square directions, is an \((n-1)\)-fold projective presentation of \(Y\), reducing to \(Y\) itself when \(n=1\). The maps are labelled as in~\eqref{Blokske}, save that the projections of the common kernel pair \(\R{d}\square\R{c}\) are now written \(d^{n-1}_{i}\) towards \(\R{c}\) and \(d^{n}_{i}\) towards \(\R{d}\). Applying the functor \(F_{n}\) to the bottom part of Figure~\ref{n-Blokske}, we find the diagram of solid arrows
\begin{equation}\label{n-H1}
    \vcenter{\xymatrix{\Ker(\overline{c_0}) \cap \Ker(\overline{c_1}) \ar@{.>}[d]_-{H_{1}F_{n}(\R{P})} \ar@{{ |>}.>}[r] &
    \overline{\R{c}} \ar@<.5ex>[r]^-{\overline{c_1}} \ar@<-.5ex>[r]_-{\overline{c_0}} \ar[d]_-{F_{n}(r)} & \overline{X} \ar[r]^-{\overline{c}} \ar[d]^-{F_{n}(d)} & \overline{C} \ar[d]^-{F_{n}(g)}\\
    0 \ar@{{ |>}.>}[r] & \R{f} \ar@<.5ex>[r]^-{f_{1}} \ar@<-.5ex>[r]_-{f_{0}} & D \ar[r]_-{f} & Z.}}
\end{equation}
Taking kernels to the left, then taking their intersection, we obtain the first homology object \(H_{1}F_{n}(\R{P})\) of the groupoid \(F_{n}(\R{P})\). Note that its codomain is trivial, since the groupoid \((\R{f},f_{0},f_{1})\) is an equivalence relation.

Taking kernels vertically now, we obtain two groupoids which happen to coincide. On the one hand, there is the internal groupoid \(\Ker(F_{n}(r))\two \Ker(F_{n}(d))\) induced by taking kernels in the diagram~\eqref{n-H1}. On the other hand, we may apply the functor \(F_{n-1}\) to the diagram in Figure~\ref{n-Blokske}, then take kernels vertically in the resulting diagram
\begin{equation}\label{n-Double}
    \vcenter{\xymatrix@=3em{F_{n-1}(\R{d}\square \R{c}) \ar@<.5ex>[r]^-{F_{n-1}(d^{n}_1)} \ar@<-.5ex>[r]_-{F_{n-1}(d^{n}_0)} \ar@<.5ex>[d]^-{F_{n-1}(d^{n-1}_1)} \ar@<-.5ex>[d]_-{F_{n-1}(d^{n-1}_0)} & F_{n-1}(\R{d}) \ar@<.5ex>[d]^-{F_{n-1}(d_1)} \ar@<-.5ex>[d]_-{F_{n-1}(d_0)} \\
    F_{n-1}(\R{c}) \ar@<.5ex>[r]^-{F_{n-1}(c_1)} \ar@<-.5ex>[r]_-{F_{n-1}(c_0)} & F_{n-1}(X).}}
\end{equation}
to find the internal groupoid
\[
    {\Ker(F_{n-1}(d^{n-1}_{1}))\two \Ker(F_{n-1}(d_{1}))}\text{.}
\]
Indeed, \(\R{d}\square\R{c}\) is the kernel pair of \(r\colon{\R{c}\to\R{f}}\), with second projection \(d^{n-1}_{1}\); the level-\(n\) form of Lemma~\ref{Lemma Kernel Extension and Kernel Pair Centralisation Comparison}, in its \(\pi_{1}\)-variant, gives \(\Ker(F_{n}(r))\cong\Ker(F_{n}(d^{n-1}_{1}))\), and since \(d^{n-1}_{1}\) is split, the level-\(n\) form of Lemma~\ref{Lemma Kernel Reflection Centralisation Comparison} gives
\[
    \Ker(F_{n}(d^{n-1}_{1}))\cong\Ker(F_{n-1}(d^{n-1}_{1}))\text{;}
\]
hence \(\Ker(F_{n}(r))\cong\Ker(F_{n-1}(d^{n-1}_{1}))\). Similarly, \(\R{d}\) is the kernel pair of the morphism \(d\colon{X\to D}\) with second projection \(d_{1}\), so that
\[
    \Ker(F_{n}(d))\cong\Ker(F_{n-1}(d_{1}))\text{.}
\]
Being induced by the universal properties of the kernels and pullbacks involved, these isomorphisms are natural, hence compatible with the horizontal groupoid structures; Lemma~\ref{Lemma Shift Functoriality} below makes this precise.

Of course the first homologies of those groupoids do now also coincide, which means that the domain \(\Ker(\overline{c_0}) \cap \Ker(\overline{c_1})\) of \(H_{1}F_{n}(\R{P})\) is
\begin{equation}\label{Higher-Intersection}
    \Ker(F_{n-1}(d_{1}^{n-1}))\cap \Ker(F_{n-1}(d_{0}^{n}))\cap \Ker(F_{n-1}(d_{1}^{n})).
\end{equation}
Lemma~\ref{Lemma H_{1} Zero Means Relation} now tells us that this intersection is nothing but
\begin{equation}\label{Higher-H2}
    \bigcap_{i\in \{n-1,n\}}(\Ker(F_{n-1}(d^{i}_{0}))\cap \Ker(F_{n-1}(d^{i}_{1}))).
\end{equation}
Indeed, by Theorem~\ref{Theorem Interpretation Fundamental Group}, the first homology of the bottom horizontal groupoid in~\eqref{n-Double} is the fundamental group of the projective object \(C\), which is trivial by the level-\((n-1)\) form of Corollary~\ref{Corollary Fundamental Group Projective}. Hence this groupoid is an equivalence relation. If now \(x\) is an element of the intersection~\eqref{Higher-Intersection}, then \(F_{n-1}(d_{0}^{n-1})(x)=0\), because
\[
    F_{n-1}(c_{0})(F_{n-1}(d_{0}^{n-1})(x))=F_{n-1}(d_{0})(F_{n-1}(d^{n}_{0})(x))=F_{n-1}(d_{0})(0)=0
\]
and
\[
    F_{n-1}(c_{1})(F_{n-1}(d_{0}^{n-1})(x))=F_{n-1}(d_{0})(F_{n-1}(d^{n}_{1})(x))=F_{n-1}(d_{0})(0)=0,
\]
while \(F_{n-1}(c_{0})\) and \(F_{n-1}(c_{1})\) are jointly monic.

By definition, the intersection \eqref{Higher-H2} is precisely \(H_{1}H_{1}(F_{n-1}(\Eq^{2}(P)))\). Thus we obtain the following result:

\begin{lemma}[Key Lemma]\label{Key Lemma}
    For \(n\geq 1\), if \(P\) is an \((n+1)\)-fold presentation of an object \(Y\) and \(F\) is a Birkhoff reflector, then
    \[
        \iota H_{1}^{2}(F_{n-1}(\Eq^{2}(P)))=\bigl(H_{1}H_{1}(F_{n-1}(\Eq^{2}(P)))\to 0\bigr)\cong H_{1}F_{n}(\R{P}).
    \]
    More generally, the same conclusion holds for any \((n+1)\)-fold extension \(P\) as in Figure~\ref{n-Blokske} where \(c\) is a split epimorphism.
\end{lemma}
\begin{proof}
    The equality is the definition of \(\iota\) (Notation~\ref{Notation iota}) together with the identification of the intersection~\eqref{Higher-Intersection} with~\eqref{Higher-H2}, and the isomorphism is established in the preceding paragraphs. For the more general statement, recall that the projectivity of \(C\) was used only to make the first homology of the bottom groupoid of~\eqref{n-Double} vanish; when \(c\) is a split epimorphism, this is taken care of by Proposition~\ref{Cech on Split Epi}, and the rest of the argument stays the same.
\end{proof}

Iterating the Key Lemma requires the following functoriality.

\begin{lemma}\label{Lemma Shift Functoriality}
    Let \(\D\) be a pointed category with finite limits. The embedding \(\iota\colon{\D\to\Arr(\D)}\) of Notation~\ref{Notation iota} preserves finite limits, and so commutes with the formation of homology of internal groupoids: \(H_{1}^{k}\circ\iota=\iota\circ H_{1}^{k}\) for every \(k\geq 1\).

    Moreover, the isomorphism of Lemma~\ref{Key Lemma} is natural in \(P\). Precisely, on the category whose objects are the \((n+1)\)-fold extensions of the shape displayed in Figure~\ref{n-Blokske} in which \(c\) is a split epimorphism---a morphism being a morphism of such diagrams---the assignments \(P\mapsto \iota H_{1}^{2}(F_{n-1}(\Eq^{2}(P)))\) and \(P\mapsto H_{1}F_{n}(\R{P})\) are functorial, and Lemma~\ref{Key Lemma} is a natural isomorphism between them. Consequently, when \(P\) carries an internal groupoid structure in one or more further directions, the isomorphism is one of internal multiple groupoids in those directions.
\end{lemma}
\begin{proof}
    The domain component of \(\iota\) is the identity and its codomain component is constant at the zero object, so \(\iota\) preserves finite limits; since iterated first homology is computed degree-wise by taking kernels and their intersections (all of which are finite limits) it is preserved by \(\iota\), which gives \(H_{1}^{k}\circ\iota=\iota\circ H_{1}^{k}\).

    For the naturality, every ingredient of the isomorphism of Lemma~\ref{Key Lemma} is functorial in \(P\): the kernel-pair functors \(\Eq\) and \(\Eq^{2}\), the reflectors \(F_{n-1}\) and \(F_{n}\), the formation of kernels and of their intersections, and the embedding \(\iota\). The one step that is not the application of a functor---the identification, in the paragraph preceding the lemma, of the groupoid \(\Ker(F_{n}(r))\two\Ker(F_{n}(d))\) with
    \[
        \Ker(F_{n-1}(d^{n-1}_{1}))\two\Ker(F_{n-1}(d_{1}))
    \]
    ---is induced by the universal properties of the kernels and pullbacks involved (Lemmas~\ref{Lemma Kernel Extension and Kernel Pair Centralisation Comparison} and~\ref{Lemma Kernel Reflection Centralisation Comparison}), and so is natural too. The final assertion, for a single further direction, follows by applying this naturality to the structure morphisms of the groupoid \(P\); several directions are treated one at a time, the isomorphism of internal groupoids obtained in one direction being itself natural, as all of its constituents are.
\end{proof}

The extensions to which the Key Lemma will be applied are faces of a higher projective presentation. By a \defn{face} of an \((n+1)\)-fold extension \(P\), viewed as an \((n+1)\)-cube, we mean any of the lower-dimensional cubes obtained by evaluating each of a number of chosen directions at its domain or codomain end. A face is again an extension: the defining condition of a higher extension~\cite[Definition~3.4]{EGVdL} is invariant under interchanging the two outer directions of \(\Arr^{n+1}(\C)\) and hence, inductively, under all permutations of the directions, so that an \((n+1)\)-fold extension may be read as an extension between \(n\)-fold extensions along any of its directions, its domain and codomain there being \(n\)-fold extensions by definition. The lemma we need extends the well-known fact that a regular epimorphism with a projective codomain is a split epimorphism.

\begin{lemma}\label{Lemma Split Faces}
    Let \(P\) be an \((n+1)\)-fold projective presentation of an object \(Y\). Then every object of \(\C\) occurring as a vertex of the \((n+1)\)-cube \(P\), with the sole exception of \(Y\) itself, is projective. Consequently, any face of \(P\) which is taken at the domain end of at least one of the directions it does not span is, when read as a morphism of cubes in any one of the directions it does span, a split epimorphism: if the face spans the directions in \(T\subsetneq n+1\) and is read in the direction \(b\in T\), then it admits a section in the category \(\Ext^{|T|-1}(\C)\) where it is an extension.
\end{lemma}
\begin{proof}
    The first statement is proved by induction on \(n\), there being nothing to prove when \(n=0\). By Definition~\ref{Definition Higher Presentation}, the domain \(\Dom P\) is extension-projective, so that all of its vertices are projective by~\cite[Lemma~5.5]{EGVdL}, while the codomain \(\Cod P\) is an \(n\)-fold projective presentation of \(Y\), whose vertices other than \(Y\) are projective by the induction hypothesis.

    For the second statement, write \(\chi\) for the face in question, read as an extension in \(\Ext^{|T|-1}(\C)\) in the direction \(b\). Every vertex of the cube \(\Cod\chi\) is a vertex of \(P\) whose coordinate in some direction outside \(T\) is a domain; in particular it is not the terminal vertex \(Y\), so it is projective by the first statement. Hence \(\Cod\chi\) is extension-projective by~\cite[Lemma~5.5]{EGVdL}. Since extension-projectivity is precisely projectivity with respect to the class of extensions~\cite[Definition~5.4]{EGVdL}, the identity on \(\Cod\chi\) lifts along \(\chi\); such a lifting is the desired section.
\end{proof}

\subsection{The main result}

\begin{proposition}\label{Proposition Shift}
    For any \(n\geq 2\) and any \((n+1)\)-fold projective presentation~\(P\) of an object \(Y\), there is an isomorphism
    \[
        \iota^{n}\HH_{n+1}F(Y)\cong H_{1}F_{n}(\R{P})\text{,}
    \]
    the left-hand side being formed from \(P\).
\end{proposition}
\begin{proof}
    We compute, applying Lemma~\ref{Key Lemma} once in each direction, from the innermost outwards:
    \begin{align*}
        H_{1}F_{n}(\R{P}) & \cong \iota H_{1}^{2}(F_{n-1}(\Eq^{2}(P)))     \\
                          & \cong \iota^{2} H_{1}^{3}(F_{n-2}(\Eq^{3}(P))) \\
                          & \qquad\vdots                                   \\
                          & \cong \iota^{n-1} H_{1}^{n}(F_{1}(\Eq^{n}(P))) \\
                          & \cong \iota^{n} H_{1}^{n+1}(F(\Eq^{n+1}(P))).
    \end{align*}
    The first isomorphism is Lemma~\ref{Key Lemma}, applied to the projective presentation \(P\) at level \(n\), the kernel pairs in \(\Eq^{2}(P)\) being taken in the directions \(n-1\) and \(n\).

    Consider then the \(j\)th isomorphism of the chain, where \(2\leq j\leq n\), and write \(k\coloneq n+1-j\) for the level at which Lemma~\ref{Key Lemma} is now applied. The source \(\iota^{j-1}H_{1}^{j}(F_{k}(\Eq^{j}(P)))\) of this isomorphism is computed from the \(j\)-fold equivalence relation \(\Eq^{j}(P)\), formed in the directions \(k\), \dots, \(n\) of \(P\), whose entries are objects of \(\Ext^{k}(\C)\); the innermost homology is the one of the direction \(k\). Along the outer directions \(k+1\), \dots, \(n\), this structure is a \((j-1)\)-fold equivalence relation whose corners are indexed by the subsets \(S\subseteq\{k+1,\dots,n\}\) recording the directions in which a corner sits at the kernel-pair level. For such a subset \(S\), let \(Q_{S}\) be the square in the directions \((k-1,k)\) obtained from the face of \(P\) spanned by the directions \(\{0,\dots,k\}\cup S\), taken at the domain end of each of the remaining directions, by forming kernel pairs in all the directions in \(S\); as in the discussion of Figure~\ref{n-Blokske}, these squares are \((k+1)\)-fold extensions, kernel pairs of higher extensions being computed pointwise~\cite[Proposition~3.5]{EGVdL}. The direction-\(k\) groupoid structure of \(F_{k}(\Eq^{j}(P))\) at the corner \(S\) is then precisely the Galois groupoid \({F_{k}(\R{Q_{S}})\two F_{k}(\Dom Q_{S})}\) of the extension \(Q_{S}\), so that the innermost homology of the source is \(H_{1}F_{k}(\R{Q_{S}})\), formed naturally in \(S\). Note that the earlier steps of the chain enter here only through the bookkeeping of directions: the squares \(Q_{S}\) are built directly out of the presentation \(P\).

    Each of the squares \(Q_{S}\) satisfies the hypothesis of the split-epimorphism case of Lemma~\ref{Key Lemma}. Indeed, the edge of \(Q_{S}\) which plays the part of \(c\) in Figure~\ref{n-Blokske} is its direction-\(k\) morphism at the domain end of the direction \(k-1\); it is obtained by applying the kernel-pair functors of the directions in \(S\) to the face of \(P\) spanned by the directions \(\{0,\dots,k-2\}\cup\{k\}\cup S\), taken at the domain end of all the other directions and read as a morphism in the direction \(k\)---when \(S\) is empty, the edge is such a face itself. By Lemma~\ref{Lemma Split Faces} that face is a split epimorphism, and a kernel-pair functor, like any functor, preserves split epimorphisms; hence the \(c\)-edge of \(Q_{S}\) is a split epimorphism, as required.

    Lemma~\ref{Key Lemma} at level \(k\) thus provides isomorphisms
    \[
        H_{1}F_{k}(\R{Q_{S}})\cong\iota H_{1}^{2}(F_{k-1}(\Eq^{2}(Q_{S}))),
    \]
    one for each corner \(S\). The structure morphisms of the \((j-1)\)-fold equivalence relation of corners are morphisms of diagrams of the shape displayed in Figure~\ref{n-Blokske} between squares with a split \(c\)-edge, so by the naturality of Lemma~\ref{Lemma Shift Functoriality} these isomorphisms form an isomorphism of \((j-1)\)-fold internal groupoids in the outer directions \(k+1\), \dots,~\(n\). Applying the iterated first homology \(H_{1}^{j-1}\) of those directions to both sides, moving the embedding \(\iota\) across it by means of the equality \(H_{1}^{j-1}\circ\iota=\iota\circ H_{1}^{j-1}\), and observing that the double kernel pairs \(\Eq^{2}(Q_{S})\) of the corners together constitute \(\Eq^{j+1}(P)\), we obtain the \(j\)th isomorphism of the chain. The result now follows from Proposition~\ref{Proposition HH_n+1 as Homology of a Groupoid}, which shows that the last object \(\iota^{n}H_{1}^{n+1}(F(\Eq^{n+1}(P)))\) is isomorphic to \(\iota^{n}\HH_{n+1}F(Y)\).
\end{proof}

\begin{theorem}\label{Theorem Comparison Level n}
    In the varietal context of Convention~\ref{Convention Comonadic Homology} (\(\C\) a semi-abelian variety and \(F\colon{\C\to \X}\) the reflector onto a subvariety) the \Tsjech\ derived functor \(\HH_{n}F(Y)\) of an object \(Y\) is isomorphic to the \((n+1)\)st homology object \(H_{n+1}(Y,F)\), for every \(n\geq 1\).
\end{theorem}
\begin{proof}
    For \(n=1\) the statement is Theorem~\ref{Theorem Comparison Level 1} together with Convention~\ref{Convention Comonadic Homology}, by which \(\HH_{1}F(Y)\cong\pi_{1}^{F}(Y)=H_{2}(Y,F)\); for \(n=2\) it is Theorem~\ref{Theorem Comparison Level 2}. Suppose now that \(n\geq 3\), and let \(P\) be an \(n\)-fold projective presentation of \(Y\). Applying Proposition~\ref{Proposition Shift} with \(n-1\) in place of \(n\) (which is legitimate, since \(n-1\geq 2\)) we obtain
    \[
        \iota^{n-1}\HH_{n}F(Y)\cong H_{1}F_{n-1}(\R{P})\cong \iota^{n-1}H_{2}(\Cod P,F_{n-1})\cong \iota^{n-1}H_{n+1}(Y,F).
    \]
    The first isomorphism is Proposition~\ref{Proposition Shift}, the object \(H_{1}F_{n-1}(\R{P})\) being the first homology of the Galois groupoid \(F_{n-1}(\R{P})\two F_{n-1}(\Dom P)\) of \(P\) at level \(n-1\). For the second isomorphism, note that \(P\) is a projective presentation of \(\Cod P\); hence Theorem~\ref{Theorem Interpretation Fundamental Group}, applied in the level-\((n-1)\) Galois structure \(\Gamma_{n-1}\), identifies this homology object with the fundamental group \(\pi_{1}^{F_{n-1}}(\Cod P)\), which is \(\iota^{n-1}H_{2}(\Cod P,F_{n-1})\) by Convention~\ref{Convention Comonadic Homology}. For the third, \(\Cod P\) is an \((n-1)\)-fold presentation of \(Y\), so that iterating Theorem~\ref{Theorem EGVdL Shift} along the tower of Galois structures---each application lowering the level by one and raising the degree by one---yields \(H_{2}(\Cod P,F_{n-1})\cong H_{n+1}(Y,F)\). The claim \(\HH_{n}F(Y)\cong H_{n+1}(Y,F)\) now follows, because the embedding \(\iota^{n-1}\) is fully faithful.
\end{proof}

\begin{corollary}\label{Corollary Higher Cech Well Defined}
    In the varietal context of Convention~\ref{Convention Comonadic Homology}, and for every \(n\geq 1\), the object \(\HH_{n}F(Y)\) is independent of the chosen \(n\)-fold projective presentation of~\(Y\): if \(P\) and \(P'\) are two such presentations, then
    \[
        H_{1}^{n}F(\Eq^{n}(P))\cong H_{n+1}(Y,F)\cong H_{1}^{n}F(\Eq^{n}(P'))\text{.}
    \]
    Hence \(\HH_{n}F\) is well defined as a functor from \(\C\) to \(\X\).
\end{corollary}
\begin{proof}
    Proposition~\ref{Proposition HH_n+1 as Homology of a Groupoid} identifies \(H_{1}^{n}F(\Eq^{n}(P))\) with the object \(\HH_{n}F(Y)\) formed from \(P\), and Theorem~\ref{Theorem Comparison Level n} identifies the latter with \(H_{n+1}(Y,F)\). The point is that the argument nowhere restricts the presentation: for \(n\geq 3\) the chain of the proof above begins with an arbitrary \(n\)-fold projective presentation \(P\) of \(Y\); for \(n=2\) the same holds of Theorem~\ref{Theorem Comparison Level 2}, since a \(2\)-fold projective presentation of \(Y\) is precisely a \(2\)-fold extension of the shape~\eqref{Blokske} whose codomain edge \(g\colon{C\to Y}\) is a projective presentation of \(Y\); and for \(n=1\) it is Theorem~\ref{Theorem Comparison Level 1}. The comonadic homology object \(H_{n+1}(Y,F)\), on the other hand, is formed from the canonical comonadic resolution of \(Y\) and involves no choice. The functoriality of \(\HH_{n}F\) in \(Y\) is then inherited, through the displayed isomorphism, from that of \(H_{n+1}({-},F)\).
\end{proof}

\begin{remark}\label{Remark Intrinsic Functoriality}
    Corollary~\ref{Corollary Higher Cech Well Defined} thus makes \(\HH_{n}F\) a functor by transport of structure along the comparison with comonadic homology. In degree one no transport is needed: a morphism \(Y\to Y'\) lifts along projective presentations, the argument of Lemma~\ref{Lemma Homotopically Equivalent} makes the induced morphism independent of the chosen lifting, and \(\HH_{1}F\) is a functor intrinsically. Whether a direct description of this kind exists in higher degrees, making the isomorphism of Theorem~\ref{Theorem Comparison Level n} natural with respect to it, is a question we leave open: the obstruction is the one discussed before Remark~\ref{Remark Classical Hopf Formula}, and the presentation-level naturality of Lemma~\ref{Lemma Shift Functoriality} would be the starting point of such a treatment.
\end{remark}

In fact, the varietal restriction is not essential to the comparison.

\begin{theorem}\label{Theorem Comparison General}
    More generally, let \(\C\) be a semi-abelian category carrying a comonad whose counit is, at each object, a regular epimorphism with projective domain (Quillen's conditions~\cite[Section~5.2]{EGVdL}) and let \(H_{n}(-,F)\) be the comonadic homology formed from it as in Convention~\ref{Convention Comonadic Homology}. Then for any Birkhoff reflector \(F\colon{\C\to \X}\) and every \(n\geq 1\), the \Tsjech\ derived functor \(\HH_{n}F(Y)\) is isomorphic to \(H_{n+1}(Y,F)\). Theorem~\ref{Theorem Comparison Level n} is the special case in which \(\C\) is a variety and the comonad is the canonical one.
\end{theorem}
\begin{proof}
    The variety hypothesis enters the proof of Theorem~\ref{Theorem Comparison Level n} only through the comonad, which serves to form the homology objects \(H_{n}(Y,F)\) and the higher projective presentations drawn from the comonadic resolution. A comonad satisfying Quillen's conditions gives us both: every object is covered by a projective object, and the resolution yields \((n+1)\)-fold projective presentations at each level of the tower~\cite[Corollary~5.7]{EGVdL}. Hence the argument applies as such. A semi-abelian variety, monadic over \(\Set\), is the prototypical such category~\cite[Example~5.3]{EGVdL}; these are the sufficient conditions referred to in the abstract. The same applies to Corollary~\ref{Corollary Higher Cech Well Defined}, so that the higher \Tsjech\ derived functors are well defined in this generality too.
\end{proof}

Note in passing that, since \(\HH_{n}F(Y)\) is formed from projective presentations alone, Theorem~\ref{Theorem Comparison General} shows that the homology objects \(H_{n+1}(Y,F)\) do not depend on the chosen comonad, but only on the projective objects it provides, in accordance with the classical fact that comonadic homology depends only on the induced projective class~\cite{Barr-Beck}.

\section*{Acknowledgement}
This paper would not have existed without the invaluable help of Dr Guram Donadze, who took the time to answer the second author's questions concerning key ideas in~\cite{Donadze-Inassaridze-Porter}, on the occasion of their 2016 encounter in Thiruvananthapuram, at the Indian Institute of Science Education and Research.

Generative AI tools (Anthropic's Claude Fable 5.1) were used to assist in revising and restructuring a draft written by the authors, with the aim of improving its presentation; they were not used to generate new mathematical results. The authors reviewed and edited all output and take full responsibility for the content of this article.


\end{document}